%% file: foliation.tex
\documentclass[notitlepage,twoside,a4paper]{amsart}
\usepackage{mathrsfs}
\usepackage{amsmath,amssymb,enumerate}
\usepackage{epsfig,fancyhdr,color}
\usepackage{epstopdf}
\usepackage{amssymb}
\usepackage{amsmath,amsthm}
\usepackage{latexsym}
\usepackage{amscd}
\usepackage{psfrag}
\usepackage{graphicx}
\usepackage{epsf}
\usepackage[latin1]{inputenc}
\usepackage[all]{xy}
\usepackage{tikz}
\usepackage{prettyref}
\usepackage{subfigure}
\usepackage{float}
\usepackage{color}
\usepackage{array}
\usepackage{cite}
\usepackage[totalwidth=17cm,totalheight=24cm]{geometry}

\newcommand{\II}{I\hspace{-0.1cm}I}
\newcommand{\III}{I\hspace{-0.1cm}I\hspace{-0.1cm}I}

\DeclareMathOperator{\trace}{tr}

\DeclareMathOperator{\Hess}{Hess}
\DeclareMathOperator{\Real}{Re}
\DeclareMathOperator{\Imagin}{Im}

\newtheorem{theorem}{\rm\bf Theorem}[section]
\newtheorem{proposition}[theorem]{\rm\bf Proposition}
\newtheorem{lemma}[theorem]{\rm\bf Lemma}
\newtheorem{corollary}[theorem]{\rm\bf Corollary}
\newtheorem{definition}[theorem]{\rm\bf Definition}
\newtheorem{remark}[theorem]{\rm\bf Remark}

\newcommand{\N}{{\mathbb N}}
\newcommand{\R}{{\mathbb R}}

\newcommand{\cCP}{{\mathcal{CP}}}
\newcommand{\cDS}{{\mathcal{DS}}}

\newcommand{\cHE}{{\mathcal{HE}}}
\newcommand{\cT}{{\mathcal T}}
\newcommand{\cML}{{\mathcal{ML}}}

\newcommand{\Mod}{\mbox{Mod}}

\def\interieur#1{\mathord{\mathop{\kern 0pt #1}\limits^\circ}}

\title[Hyperbolic ends with particles]{Hyperbolic ends with particles and grafting on singular surfaces}

\author{Qiyu Chen}
\address{Qiyu Chen:
School of Mathematics, Sun Yat-Sen University,
510275, Guangzhou, P. R. China;
University of Luxembourg, UR en Mathématiques
Maison du nombre, 6 avenue de la Fonte
L-4364 Esch-sur-Alzette, Luxembourg}
\email{chenqy0121@gmail.com}
\thanks{Q. Chen was partially supported by NSFC, No.11271378, and the International Program for Ph.D. Candidates, Sun Yat-Sen University. }

\author{Jean-Marc Schlenker}
\address{Jean-Marc Schlenker:
University of Luxembourg, UR en Mathématiques
Maison du nombre, 6 avenue de la Fonte
L-4364 Esch-sur-Alzette, Luxembourg}
\email{jean-marc.schlenker@uni.lu}
\thanks{J.-M. S. was partially supported by University of Luxembourg IRP GeoLoDim and by FNR project DynGeo, INTER/ANR/15/11211745. J.-M. S. also acknowledges support from U.S. National Science Foundation grants DMS-1107452, 1107263, 1107367 ``RNMS: GEometric structures And Representation varieties'' (the GEAR Network).}
\date{\today}

\begin{document}


    \begin{abstract}
    We prove that any hyperbolic end with particles (cone singularities along infinite curves of angles less than $\pi$) admits a unique foliation by constant Gauss curvature surfaces. Using a form of duality between hyperbolic ends with particles and convex globally hyperbolic maximal (GHM) de Sitter spacetime with particles, it follows that any convex GHM de Sitter spacetime with particles also admits a unique foliation by constant Gauss curvature surfaces. We prove that the grafting map from the product of Teichm\"uller space with the space of measured laminations to the space of complex projective structures is a homeomorphism for surfaces with cone singularities of angles less than $\pi$, as well as an analogue when grafting is replaced by ``smooth grafting''.
    \bigskip

	\end{abstract}

	\maketitle
	
    \input{foliation1.tex}

    \input{foliation2.tex}

    \input{foliation3.tex}

    \input{foliation4.tex}

    \input{foliation5.tex}

    \input{foliation6.tex}


\end{document}

%% file: foliation1.tex
\section{Introduction}

    Let $\theta=(\theta_{1},...,\theta_{n_0})\in (0,\pi)^{n_0}$. In this paper we consider an oriented closed surface $\Sigma$ of genus $g$ with $n_0$ marked points $p_{1},...,p_{n_0}$ and suppose that

    \begin{equation}\label{Topological condition}
        2\pi(2-2g)+\sum_{i=1}^{n_0}(\theta_{i}-2\pi)<0.
    \end{equation}

    This ensures that $\Sigma$ can be equipped with a hyperbolic metric with cone singularities of angle $\theta_{i}$ at the marked points $p_{i}$ for $i=1,...,n_0$ (see e.g. \cite{Troyanov,McOwen}). We denote by $\mathcal{T}_{\Sigma,\theta}$ the Teichm\"uller space of hyperbolic metrics on $\Sigma$ with fixed cone angles, which is the space of hyperbolic metrics on $\Sigma$ with cone singularities of angle $\theta_{i}$ at $p_{i}$, considered up to isotopies fixing each marked point (see more precisely Section 2.1). We also denote $\mathfrak{p}=(p_1,\cdots, p_{n_0})$, and let $\cML_\mathfrak{p}$ be the space of measured laminations on $\Sigma_{\mathfrak{p}}=\Sigma\setminus \{ p_1,\cdots, p_{n_0}\}$. It is well-known that for all $g\in \cT_{\Sigma, \theta}$, any $l\in \cML_\mathfrak{p}$ can be uniquely realized as a geodesic measured lamination on $(\Sigma,g)$.

\subsection{Hyperbolic ends with particles}

We are interested in non-complete 3-dimensional hyperbolic manifolds homeomorphic to $\Sigma\times \R$, with cone singularities of angle $\theta_i$ along $\{ p_i\}\times \R$, for all $i\in \{ 1,\cdots, n_0\}$. A relatively simple space of metrics of this type is provided by the {\em quasifuchsian metrics with particles} studied e.g. in \cite{MoS,LS14}: complete cone-manifolds containing a non-empty, compact, convex subset.

Those quasifuchsian manifolds with particles contain a smallest non-empty convex subset, called their {\em convex core}. The complement of the convex core is the disjoint union of two non-complete manifolds, each homeomorphic to $\Sigma\times (0,+\infty)$, complete on the $+\infty$ side, but bounded on the $0$ side by a concave pleated surface orthogonal to the particles. Moreover their boundary at infinity is endowed with a complex projective structure, with cone singularities of angle $\theta_i$ at the endpoint at infinity of the particle $\{ p_i\}\times (0,+\infty)$.

Here we are interested in more general hyperbolic ends with cone singularities, called {\em non-degenerate hyperbolic ends with particles}: non-complete hyperbolic manifolds homeomorphic to $\Sigma\times (0,+\infty)$, with cone singularities of angle $\theta_i$ along $\{ p_i\}\times (0,+\infty)$, complete on the $+\infty$ side, and bounded by a concave pleated surface orthogonal to the particles 
(see Definition \ref{def of hyperbolic ends with particles} for more details). We call $\cHE_\theta$ the space of those non-degenerate hyperbolic ends with particles, up to isotopy.

Our first result is a one-to-one correspondence between those hyperbolic ends and complex projective structures on $\Sigma$ with cone singularities of prescribed angle at the $p_i$.

\begin{theorem} \label{tm:projective}
For each hyperbolic end $M\in \cHE_\theta$, the boundary at infinity $\partial_\infty M$ is equipped with a complex projective structure with cone singularities of angle $\theta_i$ at the $p_i$. Conversely, any complex projective structure on $\Sigma$ with cone singularities of angle $\theta_i$ at the $p_i$ is obtained at infinity 
from a unique hyperbolic end $M\in \cHE_\theta$.
\end{theorem}

We will denote by $\cCP_\theta$ the space of complex projective structures on $\Sigma$ with cone singularities of angle $\theta_i$ at the $p_i$, considered up to isotopy fixing the marked points.

\subsection{Grafting on hyperbolic surfaces with cone singularities}

Given a hyperbolic end $M\in \cHE_\theta$, its concave pleated 
boundary is equipped with a hyperbolic metric $m$ with cone singularities of angle $\theta_i$ at the $p_i$. Moreover, it is pleated along a measured geodesic lamination $l$. We prove in Section 3.9 that its complex projective structure at infinity $\sigma$ is obtained by a grafting operation, applied along $l$ to the Fuchsian complex projective structure associated to $(\Sigma, m)$. Moreover, we will show that it follows from Theorem \ref{tm:projective} that any complex projective structure $\sigma\in \cCP_\theta$ is obtained uniquely in this manner. The following statement, extending a classical result of Thurston (see e.g. \cite[Theorem 4.1]{Dumas})
to hyperbolic surfaces with cone singularities, will be a consequence.

\begin{theorem}
  \label{tm:grafting}
The grafting map defined for non-singular hyperbolic surfaces extends to a map $Gr_\theta:\cT_{\Sigma,\theta}\times \cML_\mathfrak{p}\to \cCP_\theta$. This map is a homeomorphism.
\end{theorem}

\subsection{Foliations of hyperbolic ends with particles by $K$-surfaces}

We also prove that our non-degenerate hyperbolic ends with particles have a unique foliation by surfaces of constant (Gauss) curvature, extending a result of Labourie \cite[Theorem 1]{Lab}.

\begin{theorem}
  \label{tm:foliation}
Let $M\in \cHE_\theta$ be a non-degenerate hyperbolic end with particles. There is a unique foliation of $M$ by surfaces of constant curvature $K$ with $K$ varying from $-1$ near the concave pleated boundary to $0$ near the boundary at infinity. Moreover, for each $K\in (-1,0)$, $M$ contains a unique closed surface of constant curvature $K$.
\end{theorem}



\subsection{De Sitter spacetimes with particles}

Given a non-singular hyperbolic end $M$, there is a ``dual'' future-complete globally hyperbolic maximal de Sitter spacetime $M^d$. There are several ways to describe this duality, but one way is by noting that future-complete globally hyperbolic maximal de Sitter spacetimes are equipped with a complex projective structure at infinity (see \cite{Scannell}) that uniquely determines them. The complex projective structure defined at infinity by $M$ and $M^d$ are identical.

We extend this point of view to future-complete convex
globally hyperbolic maximal (GHM) de Sitter spacetimes with particles, as defined in Section 2.4.

We denote by $\cDS_\theta$ the space of future-complete 
convex GHM de Sitter spacetimes homeomorphic to $\Sigma\times (0,+\infty)$, with cone singularities of angle $\theta_i$ along $\{ p_i\}\times (0,+\infty)$.

\begin{theorem}
  \label{tm:projective-ds}
Any future-complete 
convex GHM de Sitter spacetime $M^d\in \cDS_\theta$ determines on $\Sigma$ a complex projective structure with cone singularities of angle $\theta_i$ at the $p_i$. Any complex projective structure $\sigma\in \cCP_\theta$ is obtained from a  unique $M^d\in \cDS_\theta$.
\end{theorem}

This result, along with Theorem \ref{tm:projective}, determines a natural map from $\cHE_\theta$ to $\cDS_\theta$ sending a hyperbolic end with particles to the unique future-complete 
convex GHM de Sitter spacetime with the same complex projective structure at infinity.

This duality extends to closed strictly concave surfaces (orthogonal to the particles) in those hyperbolic ends and closed strictly future-convex surfaces (orthogonal to the particles) in the corresponding de Sitter spacetimes.

\begin{theorem} \label{tm:dual-surfaces}
Let $M\in \cHE_\theta$ be a non-degenerate hyperbolic end with particles, and let $M^d\in \cDS_\theta$ be the dual future-complete 
convex GHM de Sitter spacetime with particles. Given a closed, strictly concave surface $S\subset M$, there is a unique strictly future-convex spacelike surface $S^d$ and a unique diffeomorphism $u:S\to S^d$ such that 
$u^*I^d=\III$  and 
$u^*\III^d=I$,
where $I,\III$ are the induced metric and third fundamental form on $S$, and 
$I^d$ and 
$\III^d$ are the induced metric and third fundamental form on $S^d$.
\end{theorem}

Conversely, given any space-like, strictly future-convex $S^d$ surface in $M^d$, there is a unique strictly concave surface $S$ in $M$ such that $S^d$ is the dual of $S$ in the sense of Theorem \ref{tm:dual-surfaces}.

\begin{proposition} \label{pr:dual}
Let $S$ be a strictly concave surface in $M$, and let $S^d$ be the dual surface in $M^d$. Then $S$ has constant curvature $K\in (-1,0)$ if and only if $S^d$ has constant curvature $K^d=K/(K+1)\in(-\infty,0)$.
\end{proposition}

\subsection{Foliation of de Sitter spacetimes with particles by $K$-surfaces}

As a consequence of Proposition \ref{pr:dual}, each foliation of a non-degenerate hyperbolic end with particles has a dual foliation of the dual future-complete 
convex GHM de Sitter space-time. We therefore obtain the following.

\begin{corollary}\label{corollary:k-foliation of dS}
Let $M^d\in \cDS_\theta$ be a future-complete 
convex GHM de Sitter spacetime with particles. There is a unique foliation of $M^d$ by surfaces of constant curvature $K^d$ with $K^d$ varying from $-\infty$ near the initial singularity to $0$ near the boundary at infinity. Moreover, for each $K^d\in (-\infty,0)$, $M^d$ contains a unique closed surface of constant curvature $K^d$.
\end{corollary}

This gives an affirmative answer to Question 6.4 in \cite{KS}, and generalizes a result about constant Gauss curvature foliation of future-complete globally hyperbolic maximal compact de Sitter spacetimes (see \cite[Theorem 2.1]{BBZ2}) to the case with particles.

\subsection{Smooth grafting on hyperbolic surfaces with cone singularities}

Constant Gauss curvature surfaces in hyperbolic ends are related to the ``smooth grafting'' map $SGr:(0,1)\times \cT\times \cT\to \cCP$, see \cite[Section 1.2]{BMS1}. The properties of $K$-surfaces in hyperbolic ends with particles as described here show that this ``smooth grafting'' map is still well-defined on hyperbolic surfaces with cone singularities of angles less than $\pi$, as a map $SGr_{\theta}$ from $(0,1)\times \cT_{\Sigma,\theta}\times \cT_{\Sigma,\theta}$ to $\cCP_\theta$.

For each $K\in(-1,0)$, we prove that the parametrization map $\phi_{K}:\cT_{\Sigma,\theta}\times \cT_{\Sigma,\theta}\to \mathcal{HE}_{\theta}$ is a homeomorphism (see Proposition \ref{prop:parametrization map}) and $\mathcal{HE}_{\theta}$ is parameterized by a homeomorphism $f_1: \mathcal{HE}_{\theta}\to\cCP_\theta$ (see Proposition \ref{prop: homeomorphism}). For each $r\in (0,1)$, we define $SGr_{\theta}(r,\cdot, \cdot)$ to be $f_1\circ \phi_K: \cT_{\Sigma,\theta}\times\cT_{\Sigma,\theta}\to\cCP_\theta$, where $K=-4r/(1+r)^2$. The applications of constant Gauss curvature foliations in hyperbolic ends with particles and smooth grafting on hyperbolic surfaces with cone singularities are outlined in Section \ref{ssc:SGr}.

This implies that for all $r\in (0,1)$, the map $SGr_{\theta}(r,\cdot, \cdot)$ is a homeomorphism from $\cT_{\Sigma,\theta}\times \cT_{\Sigma,\theta}$ to $\cCP_\theta$. We do not elaborate on this point here, since it 
follows from the same arguments as in the non-singular case, see \cite{BMS1}. The relations among all the spaces we consider throughout this paper are presented in Figure \ref{fig:main diagram}, which is a combination of Figure \ref{fig:hE}, Figure \ref{fig:dS} and Figure \ref{fig:parametrization}.

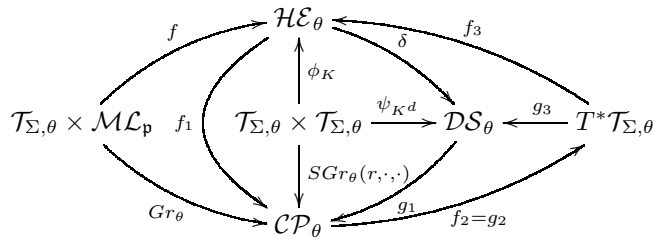
\begin{figure}
$\xymatrix{
  &\mathcal{HE}_{\theta}
  \ar@/^0.6pc/[dr]^{\delta}
  \ar@/_3pc/[dd]_{f_1}\\
 \mathcal{T}_{\Sigma,\theta}\times\mathcal{ML}_{\mathfrak{p}}
   \ar@/^0.8pc/[ur]^{f}
 \ar@/_0.8pc/[dr]_{Gr_{\theta}}
 &\mathcal{T}_{\Sigma,\theta}\times\mathcal{T}_{\Sigma,\theta}
 \ar[u]_{\phi_K}
 \ar[r]^{\quad\psi_{K^d}}
 \ar[d]^{SGr_{\theta}(r,\cdot,\cdot)}
 &\mathcal{DS}_{\theta}
 \ar@/^0.7pc/[dl]^{g_1}
 & T^*\mathcal{T}_{\Sigma,\theta}
 \ar[l]_{g_3}
   \ar@/_0.9pc/[ull]_{f_3}\\
 & \mathcal{CP}_{\theta}
  \ar@/_0.9pc/[urr]_{f_2=g_2}}$
       \caption{\small{A diagram showing the relations among all the spaces.}}
      \label{fig:main diagram}
    \end{figure}

\subsection{Outline of the paper.}

Section 2 contains the background material on various notions used in the paper.

In Section 3 we analyse the complex projective structure at infinity of a hyperbolic end with particles, and show that a hyperbolic end with particles is uniquely determined by either a complex projective structure with cone singularities, or a meromorphic quadratic differential with at worst simple poles at singularities with respect to the conformal class of a hyperbolic metric with cone singularities. We also describe the induced metric and pleating data on the ``compact'' boundary of a hyperbolic end with particles, and show that a hyperbolic end with particles is uniquely determined by a hyperbolic metric with cone singularities along with a measured lamination. As a consequence, we obtain at the end of Section 3 the proof of Theorem \ref{tm:grafting}, on the grafting map for surfaces with cone singularities.

The same analysis is conducted in Section 4 for convex GHM de Sitter spacetimes with particles. The two constructions, taken together, allow for the definition of the duality between hyperbolic ends with particles and convex GHM de Sitter spacetimes with particles, and some key properties of this duality are developed.

We then turn in Section 5 to constant Gauss curvature surfaces in hyperbolic ends with particles, and show how a pair of hyperbolic metrics with cone singularities uniquely determine a hyperbolic end with particles.

Finally, Section 6 deals with convex GHM de Sitter spacetimes with particles, develops the duality relation between hyperbolic ends with particles and convex GHM de Sitter spacetimes with particles, and obtains the results on constant Gauss curvature surfaces in those de Sitter spacetimes.


%% file: foliation2.tex
\section{Background material}

    \subsection{Hyperbolic surfaces with cone singularities.}
    First we recall the local model of a hyperbolic metric with a cone singularity of angle $\theta_{0}$ on surfaces.

    Let $\mathbb{H}^{2}$ be the Poincar\'{e} model of the hyperbolic plane. Denote by $\mathbb{H}^{2}_{\theta_{0}}$ the space obtained by taking a wedge of angle $\theta_{0}$ bounded by two half-lines intersecting at the center 0 of $\mathbb{H}^{2}$ and gluing the two half-lines by a rotation fixing 0. We call $\mathbb{H}^{2}_{\theta_{0}}$ the \emph{hyperbolic disk with cone singularity of angle $\theta_{0}$}, which is a punctured disk with the induced metric
    \begin{equation*}
        g_{\theta_{0}}=dr^{2}+\sinh^{2}(r)d\alpha^{2},
    \end{equation*}
    where $(r,\alpha)\in\mathbb{R}_{>0}\times\mathbb{R}/\theta_{0}\mathbb{Z}$ is a polar coordinate of $\mathbb{H}^{2}_{\theta_{0}}$.



    Note that the hyperbolic metrics near the cone singularities throughout this paper are assumed to satisfy a regularity condition. This ensures the existence of harmonic maps from Riemann surfaces with marked points to hyperbolic surfaces with cone singularities at the marked points (see \cite[Theorem 2]{GR}), so that we can relate minimal Lagrangian maps (see Definition \ref{def:minimal Lagrangian}) to harmonic maps, and apply the result in \cite[Lemma 3.19]{CS} to show the continuity of the parametrization map $\phi_K$ of $\mathcal{HE}_{\theta}$ (see Section 4.3).  This regularity condition is defined by using the weighted H\"older spaces (see \cite[Section 2.2]{GR} and \cite[Definition 2.1]{toulisse:minimal}).

    \begin{definition}
        For $R>0$, let $D(R):=\{z\in\mathbb{C},0<|z|<R\}$. A function $f:D(R)\rightarrow\mathbb{C}$ is said to be in $\chi^{0,\gamma}_b(D(R))$ with $\gamma\in(0,1)$ if
        \begin{equation*}
            ||f||_{\chi^{0,\gamma}_{b}}:= \sup_{z\in D(R)}|f(z)|+\sup_{z,z'\in D(R)} \frac{|f(z)-f(z')|}
            {|\alpha-\alpha'|^{\gamma}+|\frac{r-r'}{r+r'}|^{\gamma}}
            <\infty,
        \end{equation*}
        where $z=re^{i\alpha}$ and $z'=r'e^{i\alpha'}$. Let $k\in\mathbb{N}$, we say that $f\in\chi^{k,\gamma}_b (D(R))$ if $(r\partial_r)^i\partial^{j}_{\alpha}f$ is in ${\chi^{0,\gamma}_{b}}(D(R))$ for all $i+j\leq k$. In particular, this implies that $f\in \mathcal{C}^{k}(D(R))$.
    \end{definition}

    \begin{definition}
        Let $\mathfrak{p}=(p_1,...,p_{n_0})$ and $\theta=(\theta_1,...,\theta_{n_{0}})\in (0,\pi)^{n_0}$. 
        A hyperbolic metric on $\Sigma$ with cone singularities of angle $\theta$ at $\mathfrak{p}$ is a (singular) metric $g$ on $\Sigma_{\mathfrak{p}}$ with the property that for each compact subset $K\subset \Sigma_{\mathfrak{p}}$, $g|_K$ is $\mathcal{C}^{2}$ and has constant curvature $-1$, and for each marked point $p_{i}$,  there exists a neighborhood $U_{i}\subset\Sigma$ with local conformal coordinates $z$ centered at $p_i$ and a local diffeomorphism $\psi\in\chi^{2,\gamma}_{b}(U_i\setminus\{p_i\})$ such that $g|_{U_i\setminus\{p_i\}}$ is the pull back by $\psi$ of the metric $g_{\theta_i}$. We denote by $\mathfrak{M}^{\theta}_{-1}$ the space of hyperbolic metrics on $\Sigma$ with cone singularities of angle $\theta$ at $\mathfrak{p}$.
    \end{definition}

    We say that $f$ is a \emph{diffeomorphism} of $\Sigma_{\mathfrak{p}}$ if for each compact subset $K\subset \Sigma_{\mathfrak{p}}$, $f|_{K}$ is of class $\mathcal{C}^{3}$ and for each marked point $p_i$, there exists a neighbourhood $U_i\subset\Sigma$ of $p_i$ such that $f|_{U_i\setminus\{p_i\}}\in\chi^{2,\gamma}_{b}(U_i\setminus\{p_i\})$. Denote by $\mathfrak{Diff}_{0}(\Sigma_\mathfrak{p})$ the space of diffeomorphisms on $\Sigma_\mathfrak{p}$ which are isotopic to the identity (fixing each marked point). They act by pull-back on $\mathfrak{M}^{\theta}_{-1}$. We say that two metrics $h_1,h_2\in\mathfrak{M}^{\theta}_{-1}$ are \emph{isotopic} if there exists a map $f\in\mathfrak{Diff}_{0}(\Sigma_\mathfrak{p})$ such that $h_1$ is the pull back by $f$ of $h_2$.

    Denote by $\mathcal{T}_{\Sigma,\theta}$ 
    the space of isotopy classes of hyperbolic metrics on $\Sigma$ with cone singularities of angle $\theta$ at $\mathfrak{p}$. Note that $\mathcal{T}_{\Sigma,\theta}=\mathfrak{M}^{\theta}_{-1}/\mathfrak{Diff}_{0}(\Sigma_\mathfrak{p})$ and $\mathfrak{M}^{\theta}_{-1}$ is a differentiable submanifold of the manifold consisting of all $\mathcal{H}^2$ symmetric (0,2)-type tensor fields.
    $\mathcal{T}_{\Sigma,\theta}$ is a finite-dimensional differentiable manifold which inherits a natural quotient topology.

    \subsection{Hyperbolic 3-dimensional manifolds with particles.}

    First we recall the related notations and terminology in order to define hyperbolic manifolds with particles.

    \subsubsection*{Hyperbolic 3-space}
    Let $\mathbb{R}^{3,1}$ be $\mathbb{R}^4$ with the quadratic form $q(x)=x_{1}^{2}+x_{2}^{2}+x_{3}^{2}-x_{4}^{2}$. The \emph{hyperbolic 3-sapce} is defined as the quadric:
    \begin{equation*}
        {\mathbb{H}}^{3}=\{x\in\mathbb{R}^{3,1}:q(x)=-1,\, x_4>0\}.
    \end{equation*}
    It is a 3-dimensional Riemannian symmetric space of constant curvature $-1$, diffeomorphic to a 3-dimensional open ball $B^3$. The group Isom$_{0}({\mathbb{H}}^{3})$ of orientation preserving isometries of ${\mathbb{H}}^{3}$ is $SO^{+}(3,1)\cong PSL_2(\mathbb{C})$.

    \subsubsection*{The singular hyperbolic 3-space.} Let $\theta_0>0$. We define the \emph{singular hyperbolic 3-space with cone singularities of angle $\theta_0$}  as the space
    \begin{equation*}
        {\mathbb{H}}^3_{\theta_0}:=\{(\rho,r, \alpha)\in\mathbb{R}\times \mathbb{R}_{>0}\times\mathbb{R}/\theta_{0}\mathbb{Z}\}
    \end{equation*}
     with the metric
    \begin{equation*}
      d\rho^{2}+\cosh^{2}(\rho)(dr^{2}+\sinh^{2}(r)d\alpha^{2}).
    \end{equation*}
    The set $\{r=0\}$ is called the \emph{singular line} in ${\mathbb{H}}^3_{\theta_0}$ and $\theta_0$ is called the \emph{total angle} around this singular line.

     A direct computation shows that ${\mathbb{H}}^3_{\theta_0}$ has constant curvature $-1$ outside the singular line. Indeed, it is obtained from the hyperbolic plane with a cone singularity of angle $\theta_{0}$ by taking a warped product with $\mathbb{R}$ (see e.g.\cite{KS}).

    \subsubsection*{Hyperbolic manifold with particles.} A \emph{hyperbolic manifold with particles} is a 3-manifold endowed with a metric for which each point has a neighbourhood isometric to a subset of ${\mathbb{H}}^3_{\theta_0}$ for some $\theta_{0}\in(0,\pi)$.

    In a hyperbolic manifold $M$ with particles, those points which have a neighborhood isometric to a neighborhood of a point of some ${\mathbb{H}}^3_{\theta_0}$ outside the singular line are called \emph{regular points}, while the others are called \emph{singular points}. We denote by $M_r$ the set of regular points and by $M_s$ the set of singular points. By definition, $M_s$ is a disjoint union of curves. To each of those curves is associated a number, which is equal at each point to the number $\theta_0$ in the definition, called the \emph{total angle} around the singular curve (see e.g. \cite{KS,KS1,LS14}).

    \begin{definition}
   We say that $B$ is a regular half-ball in ${\mathbb{H}}^3_{\theta_0}$ if it is 
isometric to the interior of a hyperbolic half-ball in ${\mathbb{H}}^3$. 
We say that $B$ is a singular half-ball in ${\mathbb{H}}^3_{\theta_0}$ if it can be written as the subset $\{x\in {\mathbb{H}}^3_{\theta_0}: \rho>0, d(x,O)< r_0\}$ for some $r_0>0$, where $O=(0,0,0)\in {\mathbb{H}}^3_{\theta_0}$ and $d$ is the hyperbolic distance induced by the metric on ${\mathbb{H}}^3_{\theta_0}$.
    \end{definition}

    \begin{definition}\label{def of orthogonality}
        Let $S\subset {\mathbb{H}}^3_{\theta_0}$ be a surface which intersects the singular line at a point $x$. $S$ is orthogonal to the singular line at $x$ if the distance from a point $y$ of $S$ to the totally geodesic plane $P$ orthogonal to the singular line at $x$ satisfies:
        \begin{equation*}
            \lim_{y\in S, y\rightarrow x}\frac{d(y,P)}{d_S(x,y)}=0,
        \end{equation*}
        where $d_S(x,y)$ is the distance between $x$ and $y$ with respect to the induced metric on $S$.

If now $S$ is a surface in a hyperbolic manifold $M$ with particles which intersects a singular line $l$ at a point $x'$, $S$ is said to be orthogonal to $l$ at $x'$ if there exists a neighborhood $U$ of $x'$ in $M$ which is isometric to a neighborhood of a singular point in ${\mathbb{H}}^3_{\theta_0}$ such that the isometry sends $S\cap U$ to a surface orthogonal to the singular line in ${\mathbb{H}}^3_{\theta_0}$. We say that $S$ is orthogonal to the singular locus if $S$ is orthogonal to the singular curve of $M$ at each intersection with the singular locus.
    \end{definition}



\begin{definition}\label{def of concavity of surfaces}
Let $M$ be a hyperbolic manifold with particles and let $\Omega$ be a subset of the metric completion $\bar{M}$ of $M$. We say $\Omega$ is {\em concave} if there is no geodesic segment in the interior of $\Omega$ with endpoints in $\partial \Omega$.
\end{definition}

Let $M$ be a hyperbolic manifold with particles which is homeomorphic to $\Sigma\times \mathbb{R}_{>0}$ and has a metric completion $\bar{M}$ homeomorphic to $\Sigma\times \mathbb{R}_{\geq 0}$. We will write that a closed, oriented surface $S\subset \bar M$ is {\em concave} if the connected component of $\bar M\setminus S$ on the positive side is concave. We also assume that the surfaces are orthogonal to the singular locus.

It follows from the definition that if $S$ is a concave surface and $x\in S$, there is at least one ``local support plane'' of $S$ at $x$ in the neighborhood of $x$, that is, a totally geodesic disk centered at $x$ and not intersecting the negative side of $S$. In particular, if $x$ is a singular point, then the totally geodesic support disk is orthogonal to the singular curve at $x$.

\subsection{Hyperbolic ends with particles}

In this section we consider a hyperbolic manifold with particles $M$ which is homeomorphic to $\Sigma\times \mathbb{R}_{>0}$ and has a metric completion $\bar{M}$ homeomorphic to $\Sigma\times \mathbb{R}_{\geq 0}$. For convenience, we denote by $\partial_{\infty}M$ the \emph{boundary at infinity} of $M$, and by $\partial_0M$ the \emph{metric boundary} $\bar M\setminus M$, which therefore corresponds to the surface $\Sigma\times \{ 0\}$ in the identification of $\bar M$ with $\Sigma \times \R_{\geq 0}$. We will suppose that $\partial_0M$ is concave, in the sense of Definition \ref{def of concavity of surfaces}, orthogonal to the particles, and that the particles start on $\partial_0M$ and end on the boundary at infinity of $M$.

Let $x\in \partial_0M$, and let $n\in T_xM$ be a non-zero vector. We will say that $n$ is {\em normal} to $\partial_0M$ if there is a half-ball centered at $x$ in $\bar M$ such that $n$ is normal to the totally geodesic part of the boundary. We denote by $N\partial_0M$ the space of vectors normals to $\partial_0M$, so that the fiber of $N\partial_0M$ over a point where $\partial_0M$ is totally geodesic is a line, while it is an angular sector over a point of a pleating line of $\partial_0M$. Given $v=(x,n)\in N\partial_0M$, we denote by $\exp(v)\in M$ the point $\gamma(1)$, where $\gamma:[0,1]\to M$ is the geodesic such that $\gamma(0)=x$ and $\gamma'(0)=n$, if it exists. This defines a map $\exp$ from a subset of $N\partial_0M$ to $M$.

\begin{lemma} \label{lm:exp}
$\exp$ is a homeomorphism from $N\partial_0M$ to $M$.
\end{lemma}

\begin{proof}
Note first that since $\partial_0M$ is concave and $M$ is hyperbolic, $\exp$ is a local diffeomorphism from $N\partial_0M$ to $M$, sending the fibers of $N\partial_0M$ over the singular points to the cone singularities of $M$.

We will prove that $\exp$ is globally injective. Note that $\exp$ is 
injective in the neighborhood of the zero section, that is, there exists $r>0$ such that if we set
$$ N_r\partial_0M = \{ (x,n)\in N\partial_0M~|~ \| n\|<r\}~, $$
then the restriction $\exp_{|N_r\partial_0M}$ is injective. We call $r_0$ the supremum of all $r>0$ such that the restriction of $\exp$ to $N_r\partial_0M$ is injective, and we will prove that $r_0=\infty$.

Suppose by contradiction that $r_0$ is finite. It follows from the compactness of $\partial_0M$ that there exist $(x,v), (y,w)\in N\partial_0M$ such that $\|v\|=r_0$, $\| w\|\geq r_0$ and that $\exp(x,v)=\exp(y,w)$. Moreover, $\| w\|=r_0$, since otherwise the local injectivity of $\exp$ at $(x,v)$ and $(y,w)$ would imply that $\exp_{|N_r\partial_0M}$ stops being injective for $r<r_0$.

We now consider three cases, depending on whether $x$ and $y$ are regular or singular points of $\partial_0M$.
\begin{itemize}
\item If both $x$ and $y$ are singular points of $\partial_0M$, then either the cone singularities along the singular curves starting from $x$ and $y$ intersect --- this would contradict our definition of a hyperbolic manifold with particles, since the particles must be disjoint --- or those cone singularities are in fact the same singular line. In this second case, there is a singular segment of length $2r_0$ starting from $x$ and ending at $y$. This would again contradict our definition, since the particles are requested to start on $\partial_0M$ and end at infinity.
\item If both $x$ and $y$ are regular points, then the locally 
    concave surfaces $\exp(\partial(N_{r_0}\partial_0M)$ must have point of self-tangency at $\exp(x,v)=\exp(y,w)$, again by definition of $r_0$. It then follows that $\exp(\{ x\}\times [0,v])\cup \exp(\{ y\}\times [0,w])$ is a geodesic segment connecting $x$ to $y$, contradicting the concavity of $\partial_0M$.
\item If $x$ is a singular point and $y$ is regular point of $\partial_0M$. Then $\exp(\partial(N_{r_0}\partial_0M)$ intersects the singular curve starting from $y$ at $\exp(x,v)=\exp(y,w)$, and there is no such intersection for $r<r_0$. An elementary geometric argument shows that this is impossible when the cone angles are less than $\pi$, since otherwise $\exp(N_r\partial_0M)$ would already have self-intersections for $r<r_0$ close enough to $r_0$.
\end{itemize}

So we can conclude that $\exp:N\partial_0M\to M$ is globally injective. It is also proper, and since it is a local homeomorphism in the neighborhood of the zero section, we can conclude that it is a homeomorphism.
\end{proof}

     \begin{definition}\label{def of hyperbolic ends with particles}
     A non-degenerate hyperbolic end with particles is a non-complete hyperbolic manifold $M$ with particles which is homeomorphic to $\Sigma\times \mathbb{R}_{>0}$, where $\Sigma$ is a prescribed closed surface with marked points $\mathfrak{p}$, such that
     \begin{itemize}
       \item It has  a metric completion $\bar{M}$ homeomorphic to $\Sigma\times \mathbb{R}_{\geq 0}$, which is complete on the $+\infty$ side.
       \item The metric boundary $\Sigma\times \{0\}$, 
which we will denote by $\partial_0 M$, is pleated (i.e. for each $x\in \partial_0 M\setminus \bar{M}_s$, $x$ is contained in the interior of either a geodesic segment or a geodesic disk of $\bar{M}$ which is contained in $\partial_0 M$).
       \item The singular locus in $\bar{M}$ intersects $\partial_0 M$ orthogonally in totally geodesic regions.
     \end{itemize}
     \end{definition}

        The boundary at infinity $\partial_{\infty}M$ inherits a 
        complex projective structure with cone singularities (see Proposition \ref{CP1 sturctures at boundary}). The extended singular curves in $\bar{M}$ remain disjoint from each other. 

        Denote by $\mathfrak{Diff}_0(\Sigma\times \mathbb{R}_{>0})$ the space of diffeomorphisms on $\Sigma\times \mathbb{R}_{>0}$ isotopic to the identity among maps fixing each singular curve.  Two hyperbolic ends with particles $(M_1,g_1)$ and ($M_2, g_2)$ are \emph{isotopic} if there exists a map $f\in\mathfrak{Diff}_0(\Sigma\times \mathbb{R}_{>0})$ such that $g_1$ is the pull back by $f$ of $g_2$. Let $\mathcal{HE}_{\theta}$ be the space of non-degenerate hyperbolic ends with particles up to isotopy. For the sake of simplicity, we shall call the elements (as isotopy classes or their representatives) in $\mathcal{HE}_{\theta}$ \emph{hyperbolic ends with particles} henceforth.

        Let $L$ be the \emph{bending locus} of $\partial_0 M$, which is the complement of those points $x$ that admit a local support plane $P$ such that $P\cap\partial_0 M$ is a neighborhoods of $x$ in $\partial_0 M$.

   \begin{remark}\label{bending locus}
         If $L=\varnothing$, $\partial_0 M$ is totally geodesic (orthogonal to the singular locus) and we say that $M$ is Fuchsian. If $L\not=\varnothing$, it follows from the definition that $L$ is foliated by mutually disjoint complete geodesics of $\bar{M}$. Moreover, $L$ is the support of a measured lamination $\lambda$ on $\partial_0 M$, called the bending lamination, with the transverse measure recording the bending of $\partial_0M$ along $L$ (see e.g. \cite[Propositon 5.4]{BS}, \cite[Lemma A.15]{MoS}).
 \end{remark}

    Let $(M,g)$ be a hyperbolic end with particles. The shape operator $B:TS\rightarrow TS$ of an embedded surface $S\subset M$ with induced metric $I$ is defined as
    \begin{equation*}
        B(u)=\nabla_{u}n,
    \end{equation*}
    where $n$ is the positive-directed unit normal vector field on $S$ and $\nabla$ is the Levi-Civita connection of $(M,g)$. The second and third fundamental forms on $S$ are defined respectively as
    \begin{equation*}
        \II(u,v)=I(Bu,v),\qquad
        \III(u,v)=I(Bu,Bv).
    \end{equation*}

     If $S$ is smooth outside the intersection with singular locus in $M$, it is equivalent to say that $S$ is \emph{concave} (resp. \emph{strictly concave}) if the principal curvatures at each regular point of $S$ are both non-negative (resp. positive).

   \subsection{Convex GHM de Sitter spacetimes with particles}
     In order to define convex GHM de Sitter spacetimes with particles, we recall the related definitions.

     \subsubsection*{The de Sitter 3-space.}  Consider the same ambient space $\mathbb{R}^{3,1}$, similarly as for ${\mathbb{H}}^3$. The \emph{de Sitter 3-space} is defined as the quadric:
    \begin{equation*}
        {dS}_{3}=\{x\in\mathbb{R}^{3,1}:q(x)=1\}.
    \end{equation*}
    It is a 3-dimensional Lorentzian symmetric space of constant curvature $+1$, diffeomorphic to $\mathbb{S}^2\times \mathbb{R}$, where $\mathbb{S}^2$ is a 2-sphere. It is time-orientable and we choose the time orientation for which the curve $t\mapsto(\cosh t, 0,0, \sinh t)$ is future-oriented. The group Isom$_{0}({dS}_{3})$ of time-orientation and orientation preserving isometries of $dS_{3}$ is $SO^{+}(3,1)\cong PSL_2(\mathbb{C})$.

Consider the map $\pi:\mathbb{R}^{3,1}\backslash\{0\}\rightarrow \mathbb{S}^{3}$, where $\mathbb{S}^3$ is the double cover of $\R P^3$ and $\pi$ sends a point $x$ to the half-line from $0$ passing through $x$.
We define the \emph{Klein model} $\mathbb{DS}_{3}$ of de Sitter 3-space as the image $\mathbb{DS}_{3}=\pi({dS}_{3})$ (note that some authors define the Klein model as the projection of $dS_3$ to $\R P^3$ \cite[Section 2.3]{BB}, here we use $\mathbb{S}^3$ instead of $\R P^3$ in order to keep it time-orientable, see e.g. \cite[Section 5.2.1]{BBZ2}). The projection $\pi:dS_3\rightarrow \mathbb{DS}_3$ is a diffeomorphism. The boundary $\partial \mathbb{DS}_{3}$ is the image of the quadratic $Q=\{x\in\mathbb{R}^{3,1}:q(x)=0\}$ under $\pi$, which is a disjoint union of two 2-spheres: $\mathbb{S}^2_{+}=\pi(\{x\in Q:x_4>0\})$ and $\mathbb{S}^2_{-}=\pi(\{x\in Q:x_4<0\})$.

   A complete geodesic line in $\mathbb{DS}_{3}$ is \emph{spacelike} (resp. \emph{lightlike}, \emph{timelike}) if it is contained in $\mathbb{DS}_{3}$ (resp. if it is tangent to $\mathbb{S}^2_{+}$ and $\mathbb{S}^2_{-}$, if it has endpoints lying on $\mathbb{S}^2_{+}$ and $\mathbb{S}^2_{-}$ respectively).

    \subsubsection*{The singular de Sitter 3-space.} Let $\theta_0>0$. Define the \emph{singular de Sitter 3-space with cone singularities of angle $\theta_0$} as the space
    \begin{equation*}
        dS^3_{\theta_0}:=\{(t,\varphi,\alpha)\in\mathbb{R}\times [0,\pi]\times\mathbb{R}/\theta_{0}\mathbb{Z}\}
    \end{equation*}
     with the metric
    \begin{equation*}
      -dt^{2}+\cosh^{2}(t)(d\varphi^{2}+\sin^{2}(\varphi)d\alpha^{2}).
    \end{equation*}
    The set 
$\R\times \{0,\pi\}\times \mathbb{R}/\theta_{0}\mathbb{Z}$
    is called the \emph{singular line} in $dS^3_{\theta_0}$ and $\theta_0$ is called the \emph{total angle} around this singular line. One can check that $dS^3_{\theta_0}$ is a Lorentzian manifold of constant curvature $+1$ outside the singular line. Indeed, it is obtained from the spherical surface with two cone singularities of angle $\theta_0$ by taking a warped product with $\mathbb{R}$.

    An embedded surface in $dS^3_{\theta_0}$ is \emph{spacelike} if it intersects the singular line at exactly one point and it is spacelike outside the intersection with the singular locus.

    \subsubsection*{De Sitter spacetimes with particles.} A \emph{de Sitter spacetime with particles} is a (singular) Lorentzian 3-manifold in which any point $x$ has a neighbourhood isometric to a subset of $dS^3_{\theta_0}$ for some $\theta_{0}\in(0,\pi)$.

    Let $M^d$ be a de Sitter spacetime with particles which is homeomorphic to $\Sigma\times\mathbb{R}$. A closed embedded surface $S$ in $M^d$ is \emph{spacelike} if it is locally modelled on a spacelike surface in $dS^3_{\theta_0}$ for some $\theta_0\in(0,\pi)$. Similarly as the hyperbolic case, we can define the orthogonality of spacelike surfaces with respect to the singular locus in a de Sitter spacetime with particles.

    \begin{definition}
        Let $S\subset dS^3_{\theta_0}$ be a spacelike surface which intersects the singular line at a point $x$. $S$ is orthogonal to the singular line at $x$ if the causal distance from a point $y$ of $S$ to the totally geodesic plane $P$ orthogonal to the singular line at $x$ satisfies:
        \begin{equation*}
            \lim_{y\in S, y\rightarrow x}\frac{d(y,P)}{d_S(x,y)}=0,
        \end{equation*}
        where $d_S(x,y)$ is the distance between $x$ and $y$ with respect to the induced metric on $S$.

        If now $S$ is a spacelike surface in a de Sitter spacetime $M^d$ with particles which intersects a singular curve $l$ at a point $x'$. $S$ is said to be orthogonal to $l$ at $x'$ if there exists a neighborhood $U\subset M^d$ of $x'$ which is isometric to a neighborhood of a singular point in $dS^3_{\theta_0}$ such that the isometry sends $S\cap U$ to a surface orthogonal to the singular line in $dS^3_{\theta_0}$. We say that $S$ is orthogonal to the singular locus if $S$ is orthogonal to the singular curve of $M^d$ at each intersection with the singular locus.
    \end{definition}

    \begin{definition}
        Let $S$ be a spacelike surface orthogonal to the singular curves in a de Sitter spacetime with particles. We say that $S$ is future-convex if its future $I^{+}(S)$
        is geodesically convex. We say that $S$ is strictly future-convex if $I^{+}(S)$
        is strictly geodesically convex.
    \end{definition}

    \begin{definition}\label{def:ds spacetimes}
        A de Sitter spacetime $M^d$ with particles is convex GHM if
        \begin{itemize}
            \item $M^d$ is convex GH: it contains a future-convex spacelike surface $S$ orthogonal to the singular curves, which intersects every inextensible timelike curve exactly once.
            \item $M^d$ is maximal: if any isometric embedding of $M^d$ into a convex GH de Sitter spacetime is an isometry.
        \end{itemize}
    \end{definition}

     Note that by Definition \ref{def:ds spacetimes} a convex GHM de Sitter spacetime with particles is naturally future complete. Denote by $\mathfrak{Diff}_0(\Sigma\times\mathbb{R})$ the space of diffeomorphisms on $\Sigma\times\mathbb{R}$ isotopic to the identity fixing each singular line. Denote by $\mathcal{DS}_{\theta}$ the space of isotopy classes of (future-complete) convex GHM de Sitter metrics with cone singularities of angles $\theta_i$ along the singular curves $\{p_i\}\times \mathbb{R}$. Here two metrics $g_1,g_2$ are \emph{isotopic} if there exists a map $f\in\mathfrak{Diff}_0(\Sigma\times\mathbb{R})$ such that $g_1$ is the pull back by $f$ of $g_2$. For the sake of simplicity, we shall call the elements (as isotopy classes or their representatives) in $\mathcal{DS}_{\theta}$ \emph{(future-complete) convex GHM de Sitter spacetime with particles} henceforth.

    Let $(M^d,g)$ be a future-complete convex GHM de Sitter spacetime with particles. Let $S\subset M^d$ be a spacelike surface which is orthogonal to the singular locus with the induced metric $I$. The shape operator $B:TS\rightarrow TS$ of $S$ is defined as
    \begin{equation*}
        B(u)=\nabla_{u}n,
    \end{equation*}
    where $n$ is the future-directed unit normal vector field on $S$ and $\nabla$ is the Levi-Civita connection of $(M^d,g)$. The second and third fundamental forms of $S$ are defined respectively as
    \begin{equation*}
        \II(u,v)=I(Bu,v),\qquad
        \III(u,v)=I(Bu,Bv).
    \end{equation*}

    If $S$ is smooth outside the intersection with singular locus in $M^d$, it is equivalent to say that $S$ is \emph{future-convex} (resp. \emph{strictly future-convex}) if the principal curvatures at each regular point of $S$ are both non-negative (resp. positive).

    \subsection{Minimal Lagrangian maps between hyperbolic surfaces with cone singularities.}

    The construction of the parametrization of $\mathcal{HE}_{\theta}$ here depends strongly on minimal Lagrangian maps between hyperbolic surfaces with cone singularities.

    \begin{definition}\label{def:minimal Lagrangian}
        Given two hyperbolic metrics $h,h'$ on $\Sigma$  with cone singularities, a \emph{minimal Lagrangian map} $m:(\Sigma,h)\rightarrow(\Sigma,h')$ is an area-preserving and orientation-preserving diffeomorphism, sending cone singularities to cone singularities, such that its graph is a minimal surface in $(\Sigma\times\Sigma,h\oplus h')$.
    \end{definition}

    We introduce the following result (see \cite[Theorem 1.3]{Toulisse1}).

    \begin{theorem}[Toulisse]
        \label{Toulisse1}
        Let $h,h'\in \mathfrak{M}^{\theta}_{-1}$. Then there exists a unique minimal Lagrangian diffeomorphism $m:(\Sigma,h)\rightarrow (\Sigma,h')$ isotopic to the identity among maps sending each cone singularity of $h$ to the corresponding cone singularity of $h'$.
    \end{theorem}


    Minimal Lagrangian maps between hyperbolic surfaces with metrics in $\mathfrak{M}^{\theta}_{-1}$ have an equivalent description in terms of morphisms between tangent bundles (see e.g.\cite[Proposition 6.3]{Toulisse1}, \cite[Proposition 2.12]{CS}).

    \begin{proposition}
        \label{Toulisse2}
        Let $h,h'\in \mathfrak{M}^{\theta}_{-1}$, and let $m:(\Sigma,h)\rightarrow (\Sigma,h')$ be a diffeomorphism fixing each singular point. Then $m$ is a minimal Lagrangian map if and only if there exists a bundle morphism $b:T\Sigma\rightarrow T\Sigma$ defined outside the singular locus which satisfies the following properties:
        \begin{itemize}
            \item $b$ is self-adjoint for $h$ with positive eigenvalues.
            \item $\det(b)=1$.
            \item $b$ satisfies the Codazzi equation: $d^{\nabla}b=0$, where $\nabla$ is the Levi-Civita connection of $h$.
            \item $h(b\bullet,b\bullet)=m^*h'$.
            \item Both eigenvalues of $b$ tend to 1 at the cone singularities.
        \end{itemize}
    \end{proposition}

    \begin{corollary}\label{hyperbolic metrics and bundle morphism}
        Let $h,h'\in \mathfrak{M}^{\theta}_{-1}$. Then there exists a unique bundle morphism $b:T\Sigma\rightarrow T\Sigma$ defined outside the singular locus, which is self-adjoint for $h$ with positive eigenvalues, has determinant 1 and satisfies the Codazzi equation: $d^{\nabla}b=0$, where $\nabla$ is the Levi-Civita connection of $h$, such that $h(b\bullet,b\bullet)$ is isotopic to $h'$ and both eigenvalues of $b$ tend to 1 at the cone singularities. 
    \end{corollary}

 \begin{definition}\label{defintion of normalized representatives}
        We say that a pair of hyperbolic metrics $(h,h')$ is normalized if there exists a bundle morphism $b:T\Sigma\rightarrow T\Sigma$ defined outside the singular locus, which is self-adjoint for $h$, has determinant $1$, and satisfies the Codazzi equation, such that $h'=h(b\bullet,b\bullet)$, or equivalently if the identity from $(\Sigma,h)$ to $(\Sigma,h')$ is a minimal Lagrangian diffeomorphism.
    \end{definition}

    \begin{remark}\label{rk: normalized representatives}
        By Corollary \ref{hyperbolic metrics and bundle morphism}, for any $(\tau,\tau')\in\mathcal{T}_{\Sigma,\theta}\times \mathcal{T}_{\Sigma,\theta}$, we can realize $(\tau,\tau')$ as a normalized representative $(h,h')$. Note that the normalized representative of $(\tau,\tau')$ is unique up to isotopies acting diagonally on both $h$ and $h'$.
    \end{remark}

We also introduce the following proposition (see e.g. \cite[Proposition 3.12]{KS}, \cite{Lab}), which provides a convenient formula to compute the (sectional) curvatures of certain metrics.

\begin{proposition}
        \label{computation of connection and curvature}
        Let $\Sigma$ be a surface with a Riemann metric $g$. Let $A:T\Sigma\rightarrow T\Sigma$ be a bundle morphism such that $A$ is everywhere invertible and $d^{\nabla}A=0$, where $\nabla$ is the Levi-Civita connection of $g$. Let $h$ be the symmetric $(0,2)$-tensor defined by $h=g(A\bullet,A\bullet)$. Then the Levi-Civita connection of $h$ is given by
        \begin{equation*}
            \nabla^{h}_{u}(v)=A^{-1}\nabla_{u}(Av),
        \end{equation*}
         and its curvature is given by
         \begin{equation*}
            K_{h}=\frac{K_{g}}{\det(A)}.
         \end{equation*}
    \end{proposition}

%% file: foliation3.tex
\section{Hyperbolic ends with particles and complex projective structures with cone singularities}

\subsection{Complex projective structure on $\Sigma$ with cone singularities}  Let $\Sigma$ be the prescribed surface with the marked points $\mathfrak{p}=(p_1,...,p_{n_0})$ and let $\theta=(\theta_1,...,\theta_{n_0})\in (0,\pi)^{n_0}$. We first give a definition of a complex projective structure on $\Sigma$ with cone singularities of fixed angles.

\begin{definition}
Let $\theta_0>0$. We call complex cone of angle $\theta_0$, and denote by $\mathbb{C}_{\theta_0}$, the quotient of the universal covering of $\mathbb{C}\setminus\{0\}$ by a rotation of angle $\theta_0$ centered at 0.
\end{definition}

\begin{definition}\label{def:hyperbolic ends with particles}
 A complex projective structure $\sigma$ on $\Sigma$ with cone singularities of angle $\theta$ at $\mathfrak{p}$ is a maximal atlas of charts from $\Sigma_{\mathfrak{p}}$ to $\mathbb{C}P^1$ such that all transition maps are restrictions of M\"obius transformations, and for each marked point $p_i$, there exists a neighborhood $\Omega_i$ of $p_i$ in $\Sigma$ and a complex projective map $u_i: \Omega_i\rightarrow \mathbb{C}_{\theta_{i}}$ sending $p_i$ to $0$, which is a diffeomorphism from $\Omega_i\setminus\{p_i\}$ to its image.
\end{definition}

Note that in the above definition $u_i$ is uniquely determined by the complex projective structure $\sigma$ up to composition on $\mathbb{C}_{\theta_i}$ with a rotation and a homothety.

Two complex projective structures $\sigma_1$, $\sigma_2$ with prescribed cone singularities are \emph{equivalent} if there is an orientation-preserving diffeomorphism $\tau:\Sigma_{\mathfrak{p}}\rightarrow\Sigma_{\mathfrak{p}}$ isotopic to the identity that pulls back the projective charts of $\sigma_2$ to projective charts of $\sigma_1$. We denote by $\mathcal{CP}_{\theta}$ the set of equivalence classes of complex projective structures on $\Sigma$ with cone singularities of angle $\theta$ at $\mathfrak{p}$.

Each complex projective structure $\sigma$ on $\Sigma$ with prescribed cone singularities defines a local diffeomorphism from the universal covering $\widetilde{\Sigma_{\mathfrak{p}}}$ to $\mathbb{C}P^1$, which is a complex projective diffeomorphism with respect to the complex projective structure on $\widetilde{\Sigma_{\mathfrak{p}}}$ and $\mathbb{C}P^1$. We call this map $f: \widetilde{\Sigma_{\mathfrak{p}}}\rightarrow \mathbb{C}P^1$ a \emph{developing map} of $\sigma$. There is a homomorphism $\rho:\pi_1(\Sigma_{\mathfrak{p}})\rightarrow PSL_2(\mathbb{C})$, called a \emph{holonomy representation} of $\sigma$, such that $f$ is $\rho$-equivariant. In particular, the image of the small loop around each marked point $p_i$ under the holonomy $\rho$ is an elliptic element of $PSL_2(\mathbb{C})$ of angle $\theta_i$. We call $(f,\rho)$ a \emph{development-holonomy pair} and it is uniquely determined by $\sigma$ up to the \emph{$PSL_2(\mathbb{C})$-action}  by $(f,\rho)\mapsto(A\circ f,\, \rho^{A})$, where $\rho^A(\gamma)=A\,\rho(\gamma)\,A^{-1}$.


\subsection{The cotangent bundle of $\mathcal{T}_{\Sigma,\theta}$}
Note that each conformal class of a metric on $\Sigma_{\mathfrak{p}}$ with marked points admits a unique hyperbolic metric with cone singularities of angle $\theta_i$ at $p_i$ (see \cite[Theorem A]{Troyanov} and \cite{McOwen}), $\mathcal{T}_{\Sigma,\theta}$ is also identified with the space of equivalence classes of conformal structures on $\Sigma_{\mathfrak{p}}$ with marked points. Two conformal structures $c_1$ and $c_2$ on $\Sigma_{\mathfrak{p}}$ are \emph{equivalent} if there is an orientation-preserving self-homeomorphism of $\Sigma_{\mathfrak{p}}$ isotopic to the identity that pulls back the conformal charts of $c_2$ to conformal charts of $c_1$. For the sake of simplicity, we shall denote a conformal structure $c$ and its equivalence class $[c]$ by $c$.

It is known that (see \cite[Proposition 2.14]{toulisse:minimal}) for each $c\in\mathcal{T}_{\Sigma,\theta}$, the cotangent space $T_{c}^{*}\mathcal{T}_{\Sigma,\theta}$ of $\mathcal{T}_{\Sigma,\theta}$ at $c$ is the space of meromorphic quadratic differentials (with respect to the conformal structure $c$) on $\Sigma$ with at worst simple poles at the marked points.

We denote by $T^{*}\mathcal{T}_{\Sigma,\theta}$ the cotangent bundle of $\mathcal{T}_{\Sigma,\theta}$, which is a 
complex $6g-6+2n$-dimensional vector space 
 of meromorphic quadratic differentials with respect to a 
conformal structure in $\mathcal{T}_{\Sigma,\theta}$, with at worst simple poles at the marked points. 

\subsection{The complex projective structure 
at infinity of a hyperbolic end $M\in\mathcal{HE}_{\theta}$}
We show that the boundary at infinity $\partial_{\infty}M$ of a hyperbolic end $M\in\mathcal{HE}_{\theta}$ admits a complex projective structure with prescribed cone singularities.


\textbf{The model space $V_{\alpha}$}. Let $\alpha>0$ and let $\Delta_0$ be a fixed, oriented complete hyperbolic geodesic in 
 $\mathbb{H}^3$. Denote by $U$ the universal cover of the complement of $\Delta_0$ in $\mathbb{H}^3$ and denote by $V$ the metric completion of $U$, such that $V\setminus U$ is canonically identified to $\Delta_0$, which is called the \emph{singular set of $V$}. We define $V_{\alpha}$ (see e.g. \cite[Section 3.1]{MoS}) as the quotient of $V$ by the rotation of angle $\alpha$ around $\Delta_0$. The image of the singular set of $V$ under this quotient is called the \emph{singular set of $V_{\alpha}$}.

Let $M$ be a hyperbolic end with particles. It is clear that each singular point $x$ of $M$ has a neighborhood isometric to a subset of $V_{\alpha}$ with $\alpha$ equal to the total angle around the singular curve through $x$. Now we describe the geometry property of $M$ near the endpoints at infinity of the singular curves in $M$ by using the model $V_{\alpha}$, as in the following lemma. 
With Lemma \ref{lm:exp}, the argument is similar to that in \cite[Lemma 3.1, Lemma A.10]{MoS} as the particular case of non-interacting particles.

\begin{lemma}\label{lm:geometry at infinity}
For each point $p_i\in\partial_{\infty}M$ which is the endpoint at infinity of a singular curve in $M$, $p_i$ has a neighborhood $\Omega_i$ isometric to a neighborhood of one of the endpoints at infinity of $\Delta_0$ in $V_{\theta_i}$, where $\theta_i$ is the total angle around that singular curve.
\end{lemma}

As an analog of the 
complex projective structure (resp. complex projective structure with cone singularities) induced on the boundary at infinity of a hyperbolic end (resp. a quasi-fuchsian manifold with particles), a hyperbolic end with particles also induces a 
complex projective structure with cone singularities on the boundary at infinity (see e.g. \cite[Section 3.2]{MoS}).

\begin{proposition}\label{CP1 sturctures at boundary}
Let $M\in \cHE_\theta$ be a hyperbolic end with particles. Then the boundary at infinity $\partial_\infty M$ is equipped with a complex projective structure with cone singularities of angle $\theta_i$ at the $p_i$.

\end{proposition}

\begin{proof}
Consider the regular set $M_r$ of $M$ and denote its universal cover by $\widetilde{M_r}$. Let $\partial_{\infty}\mathbb{H}^3$ be the boundary at infinity of $\mathbb{H}^3$. Note that $M_r$ admits a developing map $dev: \widetilde{M_r}\rightarrow \mathbb{H}^3$, which is locally isometric projection (unique up to composition on the left by an isometry of $\mathbb{H}^3$).

We define $\partial_{\infty}\widetilde{M_r}$ as the space of equivalence classes of geodesic rays in $\widetilde{M_r}$, where two geodesic rays are equivalent if and only if they are asymptotic. Then $dev$ has a natural extension $dev:\widetilde{M_r}\cup\partial_{\infty}\widetilde{M_r}\rightarrow \mathbb{H}^3\cup\partial_{\infty}\mathbb{H}^3$, which is a local homeomorphism. Note that $\partial_{\infty}\mathbb{H}^3$ can be identified to $\mathbb{C}P^1$ and the fundamental group of $M_r$ acts on $\widetilde{M_r}$ by hyperbolic isometries which extend to $\partial_{\infty}\widetilde{M_r}$ as M\"obius transformation. We can define the boundary at infinity of $M_r$, called $\partial_{\infty}M_r$, as the quotient of $\partial_{\infty}\widetilde{M_r}$ by the fundamental group of $M_r$. Then $\partial_{\infty}M_r$ carries a canonical $\mathbb{C}P^1$-structure.

It remains to consider the behavior of the $\mathbb{C}P^1$-structure on $\partial_{\infty} M$ near the endpoints of the singular locus in $M$. By Lemma \ref{lm:geometry at infinity}, there exists a complex projective map $u_i: \Omega_i\rightarrow \mathbb{C}_{\theta_{i}}$ sending $p_i$ to $0$, which is a diffeomorphism from $\Omega_i\setminus\{p_i\}$ to its image. By Definition \ref{def:hyperbolic ends with particles}, $\partial_{\infty}M$ has a $\mathbb{C}P^1$-structure with cone singularities (at the endpoints at infinity of the singular curves) of angle equal to the total angle around the corresponding singular curve.
\end{proof}

\subsection{The meromorphic quadratic differential induced by a complex projective structure in $\mathcal{CP}_{\theta}$}
As the non-singular case, we can relate 
$\mathcal{CP}_{\theta}$ to the space $T^{*}\mathcal{T}_{\Sigma,\theta}$
 by using Schwarzian derivatives with a special analysis near the cone singularities.

Note that M\"obius transformations are biholomorphic on $\mathbb{C}P^1$ and $\mathbb{C}P^1$ admits a unique complex structure, a complex projective structure on $\Sigma$ with cone singularities also determines a complex (or conformal) structure with marked points. Note also that a hyperbolic metric on $\Sigma$ with cone singularities is a special complex projective structure on $\Sigma$ with cone singularities (the M\"obius transformations as transition functions preserve the unit circle). There is also a natural forgetful map
\begin{equation*}
\pi:\mathcal{CP}_{\theta} \rightarrow \mathcal{T}_{\Sigma,\theta},
\end{equation*}
 which is continuous and surjective. If $\sigma\in\mathcal{CP}_{\theta}$ satisfies that $\pi(\sigma)=c$, we say that $\sigma$ is a 
 complex projective structure with the \emph{underlying conformal structure} $c$. 

Let $\sigma$ be a 
complex projective structure on $\Sigma$ with prescribed cone singularities with the underlying conformal structure $c$. Let $\sigma_F$ be the hyperbolic metric on $\Sigma$ with prescribed cone singularities in the conformal class $c$. We call $\sigma_F$ the \emph{Fuchsian} complex projective structure on $\Sigma$ associated to $\sigma$ with prescribed cone singularities. Note that the union of the $\mathbb{C}P^1$-atlas of $\sigma$ and the $\mathbb{C}P^1$-atlas  of $\sigma_F$ induces a complex atlas,  the identity map $id:(\Sigma_{\mathfrak{p}},\sigma_F)\rightarrow(\Sigma_{\mathfrak{p}},\sigma)$ is a conformal map, 
 but not necessary a complex projective map. For convenience, we call this identity map the \emph{ natural conformal map}  from $\sigma_F$ to $\sigma$. Similarly, we can consider a natural conformal map from $\sigma$ to $\sigma_F$.

In the non-singular case, the Schwarzian derivative measures the ``difference'' between a pair of complex projective structures on a Riemann surface. For the singular case, we can also use this tool to measure the difference between two complex projective structures in $\mathcal{CP}_{\theta}$ with the same underlying conformal structure, but one needs to analyze the behavior of the Schwarzian derivative at the cone singularities.

Let $\Omega$ is a connected open subset of $\mathbb{C}$ and let $f: \Omega\rightarrow \mathbb{C}P^1$ be a locally injective holomorphic map. Recall that the \emph{Schwarzian derivative} of $f$ is the holomorphic quadratic differential on $\Omega$.
\begin{equation*}
\mathcal{S}(f)=\left\{\left(\frac{f''(z)}{f'(z)}\right)'-\frac{1}{2}\left(\frac{f''(z)}{f'(z)}\right)^2\right\}dz^2
\end{equation*}


Recall that the Schwarzian derivative has two important properties:
\begin{enumerate}[(1)]
  \item The Schwarzian derivative of a M\"obius transformation is zero.
  \item The cocycle property: $\mathcal{S}(g\circ f)=\mathcal{S}(f)+f^{*}\mathcal{S}(g)$, where $f^{*}\mathcal{S}(g)$ is the pull back of the holomorphic quadratic differential $\mathcal{S}(g)$ under the map $f$.
\end{enumerate}

\begin{lemma}\label{lm:cp to qd}
Let $\sigma\in\mathcal{CP}_{\theta}$ be a complex projective structure. Then the Schwarzian derivative of the conformal map $id:(\Sigma_{\mathfrak{p}},\sigma)\rightarrow (\Sigma_{\mathfrak{p}},\sigma_F)$ is a meromorphic quadratic differential in $T_{c}^{*}\mathcal{T}_{\Sigma,\theta}$, where $c$ is the common underlying conformal structure of $\sigma$ and $\sigma_F$.
\end{lemma}

\begin{proof}
Let $\varphi$ be a local expression (which is a family of locally injective holomorphic functions with respect to the $\mathbb{C}P^1$-charts of $\sigma$ and $\sigma_F$) of the map $id:(\Sigma_{\mathfrak{p}},\sigma)\rightarrow (\Sigma_{\mathfrak{p}},\sigma_F)$.
Thanks to properties (1) and (2) above, the Schwarzian derivative of $\varphi$ remains compatible with the transition functions in the overlaps of two $\mathbb{C}P^1$-charts associated to $\sigma$ or $\sigma_F$, respectively. Thus $\mathcal{S}(\varphi)$ is a holomorphic quadratic differential on $\Sigma_{\mathfrak{p}}$.

It remains to consider the behavior of $\mathcal{S}(\varphi)$ near the cone singularities. By Definition \ref{def:hyperbolic ends with particles}, for each $p_i$ on the complex projective surface $(\Sigma,\sigma)$ (resp. $(\Sigma,\sigma_F)$) with cone singularities, there is a neighborhood $\Omega_i$ (resp. $\Omega_i^F$) of $p_i$ and a complex projective map $u:\Omega_i\rightarrow\mathbb{C}_{\theta_i}$ (resp. $u_F:\Omega_i^F\rightarrow\mathbb{C}_{\theta_i}$) sending $p_i$ to 0, which is a diffeomorphism from $\Omega_i\setminus \{p_i\}$ (resp. $\Omega_i^F\setminus\{p_i\}$) to its image. Note that there is a natural holomorphic local diffeomorphism from $\mathbb{C}_{\theta_i}$ to $\mathbb{C}$, defined by sending a point $u\in\mathbb{C}_{\theta_i}$ to $u^{2\pi/\theta_i}$. We denote by $z$, $z_F$ the complex coordinates on $\Omega_i$, $\Omega_i^F$, respectively. Let $f$ be the expression of $\varphi$ near $p_i$ under these coordinates with $f(0)=0$. It is clear that $f$ is a conformal map in a small punctured neighborhood of 0  with the puncture at 0 and it can be continuously extended to the point 0. Hence $f$ is conformal in a small neighbourhood of 0 and has the expansion:
\begin{equation*}
f(z)=a_1 z+a_2 z^2 + ...+ a_n z^n+...,
\end{equation*}
where $a_1\not=0$, $a_i\in \mathbb{C}$ for $i=1,2,...$.

Then the map $\varphi$ near $p_i$ has the following expression with respect to the complex projective coordinate $u$ via the complex coordinates $z$ and $z_F$:
\begin{equation*}
\varphi(u)=(f((u)^{\frac{2\pi}{\theta_i}}))^{\frac{\theta_i}{2\pi}}.
\end{equation*}
A direct computation shows that the Schwarzian derivative $\mathcal{S}(\varphi)(u)$ has the following expansion near $u(p_i)\in\mathbb{C}_{\theta_i}$:
\begin{equation*}
\mathcal{S}(\varphi)(u)={u}^{\frac{2\pi}{\theta_i}-2}
(b_1 + b_2\, {u}^{\frac{2\pi}{\theta_i}}+...
+b_n\, {u}^{\frac{2\pi}{\theta_i}(n-1)}+...) \, d{u}^2,
\end{equation*}
where $b_i\in\mathbb{C}$ for $i\geq 1$.

In the complex coordinate $z=u^{\frac{2\pi}{\theta_i}}$, the Schwarzian derivative $\mathcal{S}(\varphi)(u)$ is expressed as
\begin{equation*}
\begin{split}
 \mathcal{S}(\varphi)\circ z^{\frac{\theta_i}{2\pi}}(z)
& =z^{1-\frac{\theta_i}{\pi}}
(b_1 + b_2\, z+\cdots +b_n\,{z^{n-1}}+...) \, \left(dz^{\frac{\theta_i}{2\pi}}\right)^2\\
& =\left(\frac{\theta_i}{2\pi}\right)^2 \frac 1{z}(b_1 + b_2\, z+ \cdots +b_n\,{z^{n-1}}+\cdots) \, d z^2.
\end{split}
 \end{equation*}

 This implies that $\mathcal{S}(\varphi)$ is a meromorphic quadratic differential on $\Sigma$ with at worst simple poles at the cone singularities, with respect to the common underlying conformal structure of $\sigma$ and $\sigma_F$. The lemma follows.
\end{proof}

\subsection{Maximal concave extension of a hyperbolic structure near infinity}

To construct a hyperbolic end with particles from a complex projective structure with cone singularities, we first prove a proposition which ensures the existence and the uniqueness (up to isometry) of the maximal extension of a hyperbolic manifold with particles which has a concave metric boundary. Moreover, we show that this maximal extension is a hyperbolic end with particles, in the sense of Definition \ref{def of hyperbolic ends with particles}.

We first introduce two definitions.

\begin{definition}\label{def of extremal points}
        Let $M$ be a hyperbolic manifold with particles. Let $S$ be a surface in $\bar{M}$. We say that a regular (resp. singular) point $x\in S$ is extremal if there exists a half-ball $B$ in $\mathbb{H}^3$ (resp. ${\mathbb{H}}^3_{\theta}$ for some $\theta_0\in(0,\pi)$), and an isometric embedding $\varphi:B \rightarrow \bar{M}$ sending the center of $B$ to $x$, such that $\varphi(\bar{B})\cap S=\{x\}$.
\end{definition}

    For example, all the points of a strictly concave surface in a hyperbolic manifold with particles are extremal points. The metric boundary $\partial_0 M$ of a hyperbolic end $M$ with particles contains no extremal points, since $\partial_0 M$ is pleated (see Definition \ref{def of hyperbolic ends with particles}).

\begin{definition}\label{def of extension}
    Let $M$ be a hyperbolic manifold with particles which has a concave metric boundary. We say $M'$ is a concave extension of $M$ if $M'$ is a hyperbolic manifold with particles such that $\partial_0 M'$ is concave and $M$ can be isometrically embedded in $M'$. We say $M'$ is a maximal concave extension of $M$ if $M'$ is 
 a concave extension of $M$ and 
 any concave extension of $M'$ is isometric to $M'$.
\end{definition}

\begin{proposition}
        \label{existence and uniqueness of maximal extension}
            Let $M_0$ be a hyperbolic manifold with particles which has a concave metric boundary. Then there exists a unique (up to isometry) 
            maximal concave extension of $M_0$, called $M$, in which $M_0$ can be isometrically embedded. Moreover, $M$ is a hyperbolic end with particles.
\end{proposition}

 \begin{proof}
    We show this proposition in the following three steps. The argument is an adaption of those for the corresponding results in globally hyperbolic spacetimes (see e.g. \cite[Theorem 3]{BGeroch}, \cite[Proposition 2.6]{BS}). The point is to use the concavity of the metric boundary of a hyperbolic manifold instead of the globally hyperbolicity of a spacetime.

   \textbf{ Step 1}:
   Let $\mathcal{E}$ be the set of all 
 concave extensions  of $M_0$. It is clear that $\mathcal{E}$ is non-empty since $M_0$ is 
 a concave extension of itself. Given $M_1,M_2 \in\mathcal{E}$, we consider the ordered pairs $(N_1, N_2)$ such that
   \begin{itemize}
     \item $N_i$ is a subset of $M_i$ in which $M_0$ can be isometrically embedded, for $i=1,2$.
     \item There is an isometric embedding from $M_0$ to $M_2$ which extends to an isometric embedding from $N_1$ to $M_2$ sending $N_1$ to $N_2$.
   \end{itemize}

   Denote by $\mathcal{C}(M_1,M_2)$ the set consisting of all such pairs for $M_1,M_2 \in\mathcal{E}$. It is clear that $\mathcal{C}(M_1,M_2)$ is partially ordered by inclusion of the first and second item of the pairs, respectively. Moreover, each totally ordered subset of $\mathcal{C}(M_1,M_2)$ has an upper bound. By Zorn's Lemma, there exists a maximal element of $\mathcal{C}(M_1,M_2)$.

   \textbf{ Step 2}:
   Now we give a partial order $``\leq"$ for the set $\mathcal{E}$ by defining $M_1\leq M_2$ if the isometric embedding from $M_0$ to $M_2$ extends to an isometric embedding from $M_1$ to $M_2$, here $M_1,M_2 \in\mathcal{E}$. We claim that $\mathcal{E}$ has a maximal element.

   Indeed, let 
   $(M_\alpha)_{\alpha\in\mathcal{A}}$  be a totally ordered subset of $\mathcal{E}$ and let $K=\sqcup_{\alpha}M_{\alpha}$ be the disjoint union of $M_{\alpha}$ over $\alpha\in\mathcal{A}$. We define an equivalence relation for the set $K$. We relate $p_{\alpha}\in M_{\alpha}$ to $p_{\beta}\in M_{\beta}$ if there exists $(N_{\alpha},N_{\beta})\in\mathcal{C}(M_{\alpha},M_{\beta})$ and an isometric embedding from $N_{\alpha}$ to $M_{\beta}$ which sends $p_{\alpha}$ to $p_{\beta}$, where $\alpha,\beta\in\mathcal{A}$. One can check that this relation is an equivalence relation on $K$.
   Denote by $\bar{K}$ the quotient space of $K$ under this equivalence relation. Then $\bar{K}$ is a manifold endowed with a natural differentiable structure and metric. Note that $M_{\alpha}\in\mathcal{E}$ and  $M_{\alpha}\subset\bar{K}$ for all $\alpha$, then $\bar{K}$ is a hyperbolic manifold with particles in which $M_0$ can be isometrically embedded.

   We claim that $\bar{K}$ has a concave metric boundary. This implies that $\bar{K}\in\mathcal{E}$ and $\bar{K}$ is an upper bound of 
   $(M_\alpha)$. Applying Zorn's Lemma again, there exists a maximal element of $\mathcal{E}$, say $M$.

   Now we show that $\bar{K}$ has a concave metric boundary.
   Note that any concave surface in a hyperbolic manifold with particles has sectional curvature at least $-1$. By the assumption \eqref{Topological condition} and the Gauss-Bonnet formula (see \cite[Proposition 1]{Troyanov}), the area of any concave surface has a positive lower bound. Note also that the area of a concave surface decreases exponentially with respect to the distance $r$ along the the normal flow pointing to the non-concave side of $S$. Combined with the fact that $M_{\alpha}$ has a concave metric boundary for all $\alpha\in\mathcal{A}$, then the metric completion of $\bar{K}$ is homeomorphic to $\Sigma\times\mathbb{R}_{\geq 0}$. Therefore, $\bar{K}$ has a metric boundary and it is naturally concave and orthogonal to the singular locus. The claim follows.

   \textbf{ Step 3}: We show that $M$ is 
   a concave extension of each element of $\mathcal{E}$. Let $M'\in\mathcal{E}$. We denote by $\hat{M}$ the quotient space of the disjoint union of $M'$ and $M$ under the equivalence relation defined above. It suffices to show $\hat{M}\in\mathcal{E}$, since this implies that $\hat{M}$ is 
   a concave extension of both $M$ and $M'$. Note that $M$ is a maximal element of $\mathcal{E}$, then $M$ is isometric to $\hat{M}$ and thus 
   a concave extension of $M'$. This shows the uniqueness of $M$ (up to isometry).

   Now we show that $\hat{M}\in\mathcal{E}$.  Let $(N',N)$ be a maximal element of $\mathcal{C}(M',M)$ (this is ensured by Step 1) and let $\psi$ be an isometric embedding from $M_0$ to $M$ which extends to an isometric embedding from $N'$ to $M$ sending $N'$ to $N$. Denote by $\partial N'$ the boundary of $N'$  in $\bar{M'}$ and denote by $\partial N$ the boundary of $N$  in $\bar{M}$. We claim that for each point $x\in\partial N'$, either $x\in\partial_0 M'$ or $\psi(x)\in\partial_0M$. 
   Otherwise, there exists a point $x\in\partial N'$ which is in the interior of $M'$ with the image $\psi(x)$ in the interior of $M$. Note that $M'$ and $M$ are both locally modelled on $\mathbb{H}^{3}_{\theta_i}$ for some $\theta_i\in(0,\pi)$. Whatever $x$ is a regular point or a singular point, we can choose a small neighborhood $U'$ of $x$ in $M'$ and a small neighborhood $U$ of $\psi(x)$ in $M$ such that they are isometric to each other. It is clear that $(N'\cup U', N\cup U)\in\mathcal{C}(M',M)$. Note also that $N'$ is a proper subset of $N'\cup U'$ in $M'$. This contradicts that $N'$ is the maximal element of $\mathcal{C}(M',M)$. The claim follows.

   Note that $\hat{M}=(N'\cong N)\sqcup (M'\setminus N')\sqcup (M\setminus N)$. Combined with the above claim, $\psi$ can extend to an isometric embedding from $\bar{N'}$ to $\bar{M}$ sending $\partial N'$ to $\partial N$, then $\hat{M}$ is Hausdorff. Note that the projection from $M\sqcup M'$ to $\hat{M}$ is open, every point of $\hat{M}$ has a neighbourhood homeomorphic to $\mathbb{R}^3$. This implies that $\hat{M}$ is a manifold. Similarly as Step 2, $\hat{M}$ inherits a natural hyperbolic structure with particles. Moreover, $\hat{M}$ can be endowed with a metric completion compatible with the metric completions of $M'$ and $M$, under which the metric boundary $\partial_0 \hat{M}$ is concave (one can check by using Definition \ref{def of concavity of surfaces}) and orthogonal to the singular locus. 
    Moreover, $\hat{M}$ is a hyperbolic manifold with particles in which $M_0$ can be isometrically embedded. This implies that $\hat{M}\in\mathcal{E}$.
\end{proof}

\subsection{The construction of hyperbolic ends in $\mathcal{HE}_{\theta}$ from meromorphic quadratic differentials in $T^{*}\mathcal{T}_{\Sigma,\theta}$}

\begin{proposition}\label{prop:qd to he}
Let $q\in T_{c}^{*}\mathcal{T}_{\Sigma,\theta}$ with $c\in\mathcal{T}_{\Sigma,\theta}$. Then there exists a unique hyperbolic end with particles $M\in\mathcal{HE}_{\theta}$ which admits a complex projective structure $\sigma$ on $\partial_{\infty}M$ in the conformal class $c$, such that the Schwarzian derivative $\mathcal{S}(\phi)$ of the natural conformal map 
$\phi:(\partial_{\infty} M,\sigma)\rightarrow (\partial_{\infty} M,\sigma_F)$ is $q$.
\end{proposition}

\begin{proof}
We construct a hyperbolic end with particles $M$ from the given quadratic differential $q$ on $\Sigma$ (with respect to the conformal structure $c$) in the following two steps.

\textbf{Step 1 :} First we construct a hyperbolic manifold with the prescribed particles $M_0$ which is homeomorphic to $\Sigma\times\mathbb{R}_{\geq0}$ with a concave metric boundary $\partial_0 M_0$.

Let $I^{*}$ be a hyperbolic metric with the prescribed cone singularities in the conformal class $c$. Let 
$\II_0^{*}=\Real \,q$ be the real part of $q$ and 
${\II}^{*}=\frac{1}{2}I^{*}+{\II}^{*}_0$. Let $B^{*}=(I^{*})^{-1}{\II}^{*}$ and ${\III}^{*}=I^{*}(B^{*}\bullet,B^{*}\bullet)$.

Let $M_0$ be the set $\Sigma\times [r_0,+\infty)$ with the metric
\begin{equation*}
g_0=dr^2+\frac{1}{2}(e^{2r}I^{*}+2{\II}^{*}+e^{-2r}{\III}^{*}),
\end{equation*}
where $r_0$ is to be determined. We claim that $M_0$ is a hyperbolic manifold with particles if we choose $r_0$ large enough. Denote $I_{r}=\frac{1}{2}(e^{2r}I^{*}+2\II^{*}+e^{-2r}{\III}^{*})$. Then we have
\begin{equation*}
\begin{split}
I_{r}&=\frac{1}{2}I^*((e^rE+e^{-r}B^{*})\bullet, (e^rE+e^{-r}B^{*})\bullet), \\
\II_r&=\frac{1}{2}\frac{dI_r}{dr}
=\frac{1}{2}I^{*}((e^rE+e^{-r}B^{*})\bullet,(e^rE-e^{-r}B^{*})\bullet)
\end{split}
\end{equation*}

Denote $B_r=(e^rE+e^{-r}B^{*})^{-1}(e^rE-e^{-r}B^{*})$. We show that $(I_r, B_r)$ satisfies the following conditions:
\begin{itemize}
 \item $B_r$ is self-adjoint for $I_r$: $I_r(B_r\bullet,\bullet)=I_r(\bullet,B_r\bullet)$. This follows directly from the fact that $B^*$ is self-adjoint for $I^*$ (since $\II^*_0$ is the real part of the quadratic differential $q$).
  \item $(I_r,B_r)$ satisfies the Gauss equation for surfaces embedded in ${\mathbb{H}}^3$: $K_{I_r}=-1+\det B_r$, where $K_{I_r}$ is the sectional curvature of $I_r$. Indeed, by the definition of $I_r$ and Proposition \ref{computation of connection and curvature},
      \begin{equation*}
      K_{I_r}=\frac{K_{{I}^{*}}}{\det (\frac{1}{\sqrt{2}}(e^rE+e^{-r}B^{*}))}=
      \frac{-2}{e^{2r}+\trace B^{*}+e^{-2r}\det B^{*}}.
      \end{equation*}
      Note that $B^{*}=(I^{*})^{-1}\II^{*}=\frac{1}{2} E+(I^{*})^{-1}\II^{*}_0$ and $(I^{*})^{-1}\II^{*}_0$ is traceless. We have $\trace B^{*}=1$ and
      \begin{equation*}
      -1+\det B_r=\frac{-2\trace B^{*}}{e^{2r}+\trace B^{*}+e^{-2r}\det B^{*}}=\frac{-2}{e^{2r}+\trace B^{*}+ e^{-2r}\det B^{*}}=K_{I_r}.
      \end{equation*}

  \item $(I_r,B_r)$ satisfies the Codazzi equation: $d^{\nabla^{I_r}}B_r=0$, where $\nabla^{I_r}$ is the Levi-Civita connection of $I_r$.
      Denote by $\nabla^{I^{*}}$ the Levi-Civita connection of $I^{*}$. By Proposition \ref{computation of connection and curvature}, 
      \begin{equation*}
      \nabla^{I_r}=(e^rE+e^{-r}B^{*})^{-1}{\nabla}^{I^{*}}(e^rE+e^{-r}B^{*})~.
      \end{equation*}
      It suffices to show that $d^{{\nabla}^{I^{*}}}B^{*}=0$. By the definition of ${\nabla}^{I^{*}}$, it can be checked that $d^{\nabla^{I^{*}}}I^{*}=0$. Note that 
      $\II_0^{*}= \Real \,q$ with $q$ a holomorphic quadratic differential outside the marked points. Then $d^{\nabla^{I^{*}}}\II_0^{*}=0$. Therefore, $d^{\nabla^{I^{*}}}B^{*}
      =d^{\nabla^{I^{*}}}\big(\frac{1}{2}E+(I^{*})^{-1}\II^{*}_0\big)
      =(I^{*})^{-1}d^{\nabla^{I^{*}}}{\II_0}^{*}$ = 0.
 \item ($I_r$, $B_r$) satisfies the following equality:
 \begin{equation*}
 I_{r+s}=I_r((\cosh(s)E+\sinh(s)B_r)\bullet,  (\cosh(s)E+\sinh(s)B_r)\bullet),
 \end{equation*}
 for all $r, s>0$. This follows from a direct computation.

\end{itemize}

     Denote by $\lambda^*$, $\mu^*$ (resp. $\lambda_r$, $\mu_r$) the eigenvalues of $B^{*}$ (resp. $B_r$). By computation,
     \begin{equation*}
     \lambda_r=\frac{e^r-e^{-r} \lambda^*}{e^r+e^{-r} \lambda^*},
      \quad\quad
     \mu_r=\frac{e^r-e^{-r} \mu^*}{e^r+e^{-r} \mu^*}.
     \end{equation*}

If $r_0$ is large enough, the eigenvalues $\lambda_{r_0}$, $\mu_{r_0}$ of $(\Sigma, I_{r_0})$ are both positive. Combined with the above properties of $(I_r, B_r)$, this shows that $M_{0}$ is a hyperbolic manifold with particles which has a concave metric boundary.

We now show that the total angle around the singular curve $\{p_i\}\times[r_0,+\infty)$ of $M_{0}$ is $\theta_i$. It suffices to check that $(\Sigma\times\{r\}, I_r)$  has a cone singularities of angle $\theta_i$ at the intersection with the singular line through $p_i$. Note that $I_{r} =\frac{1}{2}I^{*}((e^rE+e^{-r}B^{*})\bullet, (e^rE+e^{-r}B^{*})\bullet)$. We claim that $B^{*}$ tends to $\frac{1}{2}E$ at the cone singularities.
Indeed $I^{*}=\rho(z) |dz|^2$ with $\rho(z)=e^{2u} |z|^{2(\frac{\theta_i}{2\pi}-1)}$ near the cone singularity $p_i$, while the quadratic differential $q=f(z)dz^2$ has at most simple pole at $p_i$ (that is, $|f(z)|\leq O(1/|z|)$ near $z(p_i)=0$). A direct computation shows that
     \begin{equation*}
      (I^{*})^{-1}\II_0^{*}= \frac{1}{2} \, \rho^{-1}(z) \left(
                            \begin{array}{cc}
                                \Real f & -\Imagin f \\\\
                                -\Imagin f& -\Real f \\
                            \end{array}
                        \right).
     \end{equation*}
Combined with the observation that $\theta_i\in(0,\pi)$ and $|\Real f|, |\Imagin f|\leq |f|\leq O(1/|z|)$ near $z(p_i)=0$, we have that $(I^{*})^{-1}\II_0^{*}$ tends to the zero matrix at $p_i$. This implies that $B^{*}$ tends to $\frac{1}{2}E$ at $p_i$. Hence, $I_r$ tends to $\frac{1}{2}(e^r+\frac{1}{2}e^{-r})^2 I^{*}$ at $p_i$, which implies that $I_r$ has the cone singularities of the same angle $\theta_i$ at $p_i$ as those associated to $I^{*}$.

\textbf{Step 2 :} We construct the desired hyperbolic end $M$ with particles via $M_0$.

Indeed, by Proposition \ref{existence and uniqueness of maximal extension}, $M_0$ admits a unique maximal concave extension which is a hyperbolic end with particles, say $M$. We will show that the induced complex projective structure $\sigma$ on $\partial_\infty M$ satisfies the required condition.

A direct computation shows that $I^{*}=\frac{1}{2}e^{-2r}{G_r}_{*}(I_r+2\II_r+\III_r)$ (see e.g. \cite[Lemma 5.1]{KS2}), where $G_r$ is the Gauss map from $(\Sigma\times\{r\}, I_r)$ to $\partial_{\infty}M$. This implies that the conformal structure induced on $\partial_{\infty}M$ by the hyperbolic metric on $M$ is $c$. By \cite[Lemma 8.3]{KS2}, the real part of the Schwarzian derivative of the natural map 
$\phi:(\partial_{\infty} M,\sigma)\rightarrow (\partial_{\infty} M,\sigma_F)$ is $\II_0^{*}$ (note that the proof of this lemma is purely local, and therefore extends to the singular setting), where $\sigma$ is the complex projective structure induced on $\partial_{\infty} M$ and $\sigma_F$ is the Fuchsian complex projective structure of $\sigma$. Hence, 
$\Real \mathcal{S}(\phi)=\II^{*}_0=\Real q$.
This implies that $\mathcal{S}(\phi)=q$.
\end{proof}

\emph{Proof of Theorem \ref{tm:projective}} : Note that the hyperbolic end with particles in Proposition \ref{prop:qd to he} is unique from the construction. Combined with Proposition \ref{CP1 sturctures at boundary} and Lemma \ref{lm:cp to qd}, Theorem \ref{tm:projective} follows.



\subsection{Hyperbolic ends with particles in terms of the bending data on the metric boundary}

Now we consider the relation between $\mathcal{HE}_{\theta}$ and $\mathcal{T}_{\Sigma,\theta}\times\mathcal{ML}_{\mathfrak{p}}$.

\begin{proposition}\label{he and ml}
The map sending a hyperbolic end with particles to the induced metric and measured bending lamination on its metric boundary is a 
bijection between $\mathcal{HE}_{\theta}$ and $\mathcal{T}_{\Sigma,\theta}\times\mathcal{ML}_{\mathfrak{p}}$.
    \end{proposition}

    \begin{proof}
    Let $M$ be a hyperbolic end with particles. It follows from Remark \ref{bending locus} that $\partial_0M$ has a bending lamination, say $\lambda$.

    Note that the singular lines are orthogonal to $\partial_0 M$ and the total angles around the singular curves are less than $\pi$. The distance from the singular points in $\bar{M}$ to the support $L$ of the bending lamination is bounded away from 0. In particular, if $x\in\partial_0 M$ is a singular point, then $\partial_0 M$ has a local support plane at $x$ in $\bar{M}$, say $P$, such that $P\cap\partial_0 M$ contains a neighbourhood of $x$ in $P$.

     Using these facts, it follows that $\partial_0M$ can be locally isometrically embedded into a complete pleated surface in $\mathbb{H}^{3}$ (resp. a totally geodesic plane orthogonal to the singular line in $\mathbb{H}^{3}_{\theta_i}$ for some $\theta_i$) away from the singular points (resp.  near each singular point). Therefore, $\partial_0 M$ carries an intrinsic hyperbolic metric, say $h$, with cone singularities (at the intersections with singular locus) of angle equal to the total angle around the corresponding singular curve. Thus we obtain (up to isotopy) the pair $(h,\lambda)\in\mathcal{T}_{\Sigma,\theta}\times\mathcal{ML}_{\mathfrak{p}}$.

     Conversely, we will show that given a hyperbolic metric $h\in\mathcal{T}_{\Sigma,\theta}$ and a measured lamination $\lambda\in\mathcal{ML}_{\mathfrak{p}}$, there is a unique hyperbolic end with particles, say $M$, such that $h$ and $\lambda$ are the induced metric and bending lamination on $\partial_0M$. The argument is similar to that in \cite[Propositon 5.8]{BS} which considers the case of AdS manifolds with particles.

     Denote by $\widetilde{\Sigma_{\mathfrak{p}}}$ the universal cover of $\Sigma_{\mathfrak{p}}$. We claim that $h$ and $\lambda$ determine a local isometric embedding $dev_{\lambda}:\widetilde{\Sigma_{\mathfrak{p}}}\rightarrow\mathbb{H}^3$, which is equivariant under a homomorphism $\rho_{\lambda}:\pi_1(\Sigma_{\mathfrak{p}})\rightarrow PSL_2(\mathbb{C})$. Indeed, associated to $\lambda$ we can define a bending cocycle $\beta_{\lambda}:\widetilde{\Sigma_{\mathfrak{p}}}\times
     \widetilde{\Sigma_{\mathfrak{p}}}\rightarrow PSL_2(\mathbb{C})$  (see \cite[Chapter 4.1]{BB} and \cite[Definition II 3.5.2]{CME}), which satisfies the following two equalities:
     \begin{equation*}
     \begin{split}
     &\beta_{\lambda}(x,y)\circ\beta_{\lambda}(y,z)=\beta_{\lambda}(x,z),\\
     &\beta_{\lambda}(\gamma x,\gamma y)=\rho(\gamma)\beta_{\lambda}(x,y)\rho(\gamma)^{-1},
     \end{split}
     \end{equation*}
     where $\rho:\pi_1(\Sigma_{\mathfrak{p}})\rightarrow PSL_2(\mathbb{R})\leq PSL_2(\mathbb{C})$ is the holonomy representation of $h$.

     In particular, the map $dev_{\lambda}$ can be expressed in terms of $\beta_{\lambda}$, that is,
     \begin{equation*}
     dev_{\lambda}(x)=\beta_{\lambda}(x_0,x)I(dev(x)),
     \end{equation*}
     where $x_0\in\widetilde{\Sigma_{\mathfrak{p}}}$ is a fixed point, $dev$ is the developing map of $h$, and $I$ is the isometric embedding of $\mathbb{H}^2$ into $\mathbb{H}^3$. We define $\rho_{\lambda}:\pi_1(\Sigma_{\mathfrak{p}})\rightarrow PSL_2(\mathbb{C})$ as
     \begin{equation*}
     \rho_{\lambda}(\gamma)=\beta_{\lambda}(\gamma x_0,x_0)\circ \rho(\gamma),
      \end{equation*}
      for all $\gamma\in\pi_1(\Sigma_{\mathfrak{p}})$.

     One can check that $dev_{\lambda}$ is locally injective and it is $\rho_{\lambda}$-equivariant. Note that as the singular locus of $h$ on $\Sigma$ stay away from $\lambda$, the cocycle $\beta_{\lambda}(x_0,x)$ is trival in $\pi^{-1}(U_i)$ for a neighborhood $U_i$ of a marked point $p_i\in\mathfrak{p}$, where $\pi:\widetilde{\Sigma_{\mathfrak{p}}}\rightarrow \Sigma_{\mathfrak{p}}$ is the universal cover. This implies that the map $dev_{\lambda}$ is conjugated to $dev$ in $\pi^{-1}(U_i)$. Let $S$ be the surface equipped with the developing map $dev_{\lambda}$ and the holonomy representation $\rho_{\lambda}$. Then $S$ admits  a hyperbolic metric on $\Sigma_{\mathfrak{p}}$ with cone singularities of the same angle as $h$ at $\mathfrak{p}$, and bending along $\lambda$ (in terms of the local chart in $\mathbb{H}^3$ given by $(dev_{\lambda},\rho_{\lambda})$-data) with the bending angle equal to the corresponding transverse measure. 
Let us denote by $S_r$ the regular set of $S$ and by $\widetilde{S_r}$ the universal cover of $S_r$. Then $dev_{\lambda}:\widetilde{S_r}\to \mathbb{H}^3$ is a $\rho_{\lambda}$-equivariant developing map of $S_r$. Now we consider the normal exponential map, called $\exp$, of $dev_{\lambda}(S_r)\subset\mathbb{H}^3$.
\begin{equation*}
     \exp: N(dev_{\lambda}(S_r))\rightarrow \mathbb{H}^3,
\end{equation*}
where $N(dev_{\lambda}(S_r))$ is the set of the pairs $(x,v)$ such that $v$ is a locally concave-directed vector at a point $x\in dev_{\lambda}(S_r)$ for which the totally geodesic disk orthogonal to $v$ at its center $x$ is a support disk of the image under $dev_{\lambda}$ of  a neighborhood $\widetilde{U}_{\widetilde{x}}\subset\widetilde{S_r}$ of a point $\widetilde{x}\in dev_{\lambda}^{-1}(x)$ such that $dev_{\lambda}|_{\widetilde{U}_{\widetilde{x}}}$ is homeomorphic, and $\exp(x,v)=\exp_x(v)$. Note that $dev_{\lambda}(S_r)$ is locally concave in $\mathbb{H}^3$ and then $\exp$ is well-defined and indeed a local homeomorphism by construction. Hence $dev_{\lambda}(S_r)$ inherits a natural metric from the hyperbolic metric on $\mathbb{H}^3$.

Note also that the holonomy representation $\rho_{\lambda}$ for $S_r$ induces a natural action on  $N(dev_{\lambda}(S_r))$: for any $(x,v)\in N(dev_{\lambda}(S_r))$ and $\gamma\in \pi_1(S_r)$, we define $\rho_{\lambda}(\gamma)(x,v)=(\rho_{\lambda}(\gamma)(x), \rho_{\lambda}(\gamma)_{*}(v))$, where $\rho_{\lambda}(\gamma)_{*}(v)$ is the put-forward vector at $\rho_{\lambda}(\gamma)(x)$ by $\rho_{\lambda}(\gamma)$ of the vector $v$ at $x$. Now we define an identification on $dev_{\lambda}(S_r)$ by identifying $dev_{\lambda}(x,v)$ with $dev_{\lambda}(x',v')$ if $(x,v)$ is related to $(x',v')$ by an action induced by $\rho_{\lambda}(\gamma)$ for some $\gamma\in\pi_1(S_r)$. One can check that the quotient of $dev_{\lambda}(S_r)$ by this identification is a hyperbolic manifold homeomorphic to $S_r\times(0,+\infty)$ (since $dev_{\lambda}$ is locally homeomorphic and $\rho_{\lambda}$-equivariant, $\exp$ is locally homeomorphic, and the induced metric on $dev_{\lambda}(S_r)$ is invariant under this identification).

Let $M$ be the metric completion of this quotient manifold. Observe that for each small loop $\gamma_i\in\pi_1(S_r)$ near the marked point $p_i$, $\rho_{\lambda}(\gamma_i)$ is an elliptic element in $PSL_{2}(\mathbb{C})$ of angle $\theta_i$ up to conjugation. Note also that the distance from the support of $\lambda$ to the cone singularities of $S$ is bounded away from 0. Then the small neighborhood of the line $l_i=\{p_i\}\times(0,+\infty)$ in $M$ is locally modelled on $\mathbb{H}^3_{\theta_i}$, thus $l_i$ is a singular curve in $M$ with cone singularities of angle $\theta_i$ at each point.  Therefore,  $M$ is a hyperbolic end with particles in $\mathcal{HE}_{\theta}$, which has a concave pleated boundary (identified to $S$) with the induce metric $h$ and the bending lamination $\lambda$.

Let $f:\mathcal{T}_{\Sigma,\theta}\times\mathcal{ML}_{\mathfrak{p}} \rightarrow\mathcal{HE}_{\theta}$ be the map constructed above. It follows from the construction that  $f$ is well-defined, with the inverse as exactly the induced hyperbolic metric and bending lamination on $\partial_0M$. This completes the proof.
    \end{proof}

\subsection{Comparing parameterizations of $\mathcal{HE}_{\theta}$}

We now sum up the various parameterizations of the space of hyperbolic ends with particles, and the relations 
among them.

\begin{proposition}\label{prop: homeomorphism}
The following maps are homeomorphisms.
\begin{itemize}
\item The map $f:\mathcal{T}_{\Sigma,\theta}\times\mathcal{ML}_{\mathfrak{p}}\to \mathcal{HE}_{\theta}$ sending $(m,l)$ to the unique hyperbolic end with particles such that the induced metric and measured bending lamination on the metric boundary are $m$ and $l$, see Proposition \ref{he and ml},
\item the map $f_1:\mathcal{HE}_{\theta}\to \mathcal{CP}_{\theta}$ sending a hyperbolic end with particles to the complex projective structure at infinity, see Proposition \ref{CP1 sturctures at boundary},
\item the map $f_2:\mathcal{CP}_{\theta}\to T^{*}\mathcal{T}_{\Sigma,\theta}$ sending a complex projective structure to the Schwarzian derivative of its map to the Fuchsian complex projective structure with the same underlying complex structure, see Lemma \ref{lm:cp to qd},
\item the map $f_3:T^{*}\mathcal{T}_{\Sigma,\theta}\to \mathcal{HE}_{\theta}$ reconstructing a hyperbolic end with particles from the data of a hyperbolic metric and a traceless Codazzi tensor on the boundary at infinity, see Proposition \ref{prop:qd to he}.
\end{itemize}
Moreover, the triangle on the right-hand side of Figure \ref{fig:hE} commutes.
\end{proposition}

\begin{proof}
It is sufficent to show the continuity of the maps $f$, $f^{-1}$, $f_1$, $f_2$, $f_3$ in the following diagram.

    \begin{figure}[ht]
      $\xymatrix{
      \mathcal{T}_{\Sigma,\theta}\times\mathcal{ML}_{\mathfrak{p}}
      \ar[r]^{\quad \quad f}
      &\mathcal{HE}_{\theta} \ar[r]^{f_1} &\mathcal{CP}_{\theta}\ar[dl]^{f_2}\\
      &T^{*}\mathcal{T}_{\Sigma,\theta} \ar[u]^{f_3} }$
       \caption{\small{A diagram showing the relations among several spaces related to $\mathcal{HE}_{\theta}$.}}
      \label{fig:hE}
    \end{figure}
Note that the induced metric and the bending lamination on $\partial_0M$ of a hyperbolic end $M$ with particles are completely determined by the intrinsic  geometry of $M$. Conversely, a hyperbolic end with particles is
obtained as the image under the exponential map $\exp$ (defined in the proof of Proposition \ref{he and ml}) of the normal bundle $NS$, which depends continuously on the $(dev_{\lambda},\rho_{\lambda})$-data determined by the bending data $(h,\lambda)\in \mathcal{T}_{\Sigma,\theta}\times\mathcal{ML}_{\mathfrak{p}}$. Therefore, $f$ and $f^{-1}$ are naturally continuous.

As for the map $f_1$, observe that the 
complex projective structure induced on $\partial_{\infty}M$ is determined by the canonical 
complex projective structure on $\partial_{\infty}M_r$  (considered as an extended $\big(PSL_2(\mathbb{C}), \partial_{\infty}\mathbb{H}^3\big)$-structure  on $\partial_{\infty}M_r$, which depends continuously on the $\big(PSL_2(\mathbb{C}), \mathbb{H}^3\big)$-structure on $M_r$) and the asymptotic geometry near the endpoints at infinity of the singular curves in $M$ (see Lemma \ref{lm:geometry at infinity}, which ensures that the complex projective structure at infinity has cone singularities of angle $\theta_i$ at the endpoint at infinity of the singular curve $\{p_i\}\times(0,+\infty)$).  Hence, $f_1$ is naturally continuous.

A well-known fact in complex analysis says that uniformly convergent  holomorphic maps have uniformly convergent derivatives of arbitrary order  (on compact subsets). Note also that the natural maps from a complex projective structure with cone singularities to the corresponding Fuchsian complex projective structure extend conformally to the marked points (with respect to the complex charts) and there is a natural holomorphic local diffeomorphism from the $\mathbb{C}P^1$-chart in $\mathbb{C}_{\theta_i}$ to the complex chart in $\mathbb{C}$ at the singular point $p_i$  (see e.g. Lemma \ref{lm:cp to qd}). Therefore, the Schwarzian derivative induces a continuous map on the space of the natural conformal maps from a complex projective structure $\sigma$ with cone singularities to the corresponding Fuchsian complex projective structure $\sigma_F$. Moreover, the sequence of natural conformal maps $\varphi_n: (\Sigma, \sigma_n)\rightarrow (\Sigma, (\sigma_{n})_F)$ converges to the natural conformal map $\varphi: (\Sigma, \sigma)\rightarrow (\Sigma, \sigma_{F})$ (with respect to the $\mathbb{C}P^1$-charts) as $\sigma_n$ converges to $\sigma$ in $\mathcal{CP}_{\theta}$ (under the topology defined using development-holonomy pairs). It follows that $f_2$ is continuous.

Recall the proof of Proposition \ref{prop:qd to he} that the geometry of the obtained hyperbolic end $M$ with particles from a given quadratic differential $q\in T^{*}\mathcal{T}_{\Sigma,\theta}$ is completely determined by the first and second fundamental form $I^{*}$, $\II^{*}$ (defined by $q$) on $\partial_{\infty}M$. More precisely, $I^{*}$ is the hyperbolic metric with cone singularities of fixed angles in the conformal class of the underlying conformal structure of $q$ and 
$\II^{*}=\frac{1}{2}I^{*}+\Real \,q$. This implies that $I^{*}$ and $\II^{*}$ depend continuously on $q\in T^{*}\mathcal{T}_{\Sigma,\theta}$. As a result, we obtain the continuity of $f_3$.

Combining the above results, any two spaces in Figure \ref{fig:hE} are homeomorphic.
\end{proof}

\subsection{The grafting map on hyperbolic surfaces with prescribed cone singularities}
 In non-singular case, it was proved by Thurston that the grafting map $Gr:\mathcal{T}\times \mathcal{ML}\rightarrow\mathcal{CP}$ is a homeomorphism (see e.g. \cite{Dumas,KT}), where $\cT$ denotes the Teichm\"uller space of a closed oriented surface $S$ of genus at least 2, $\cML$ denotes the space of measured laminations on $S$ and $\mathcal{CP}$ is the space of complex projective structures on $S$, up to isotopy. Here we generalize this result to hyperbolic surfaces with cone singularities of angles less than $\pi$ by showing that the grafting map is indeed the composition of the maps $f$ and $f_1$ in Proposition \ref{prop: homeomorphism}.

 Recall that for a hyperbolic surface with cone singularities $p_i$ of angles $\theta_i\in(0,\pi)$, each $p_i$ has a neighborhood of a radius $r_i=r(\theta_i)>0$ (depending only on $\theta_i$) which is disjoint from any simple closed geodesic (see \cite[Theorem 3]{DP}). Note also that the weighted multicurves are dense in $\mathcal{ML}_{\mathfrak{p}}$. Then the distance from the support of any measured laminations in $\mathcal{ML}_{\mathfrak{p}}$ to $\{p_1,...,p_{n_0}\}$ has a uniformly positive lower bound. Therefore, the grafting operation can be naturally generalized to the case with cone singularities.

Let $S$ be a hyperbolic surface with the metric $h\in\cT_{\Sigma,\theta}$ and let $t\gamma$ be a $t$-weighted  simple closed geodesic on $S$. We perform a \emph{grafting operation}: cut $S$ open along $\gamma$ and glue a cylinder $\gamma\times[0,t]$ along the cutting on both side. For a disjoint union $\cup_i t_i\gamma_i$ of weighted simple closed geodesics, we can also perform this operation for each weighted geodesic $t_i\gamma_i$.  Note that this operation is done outside the union of the neighborhood $U_{r_i}$ of each singular point $p_i$ on $S$ with a radius $r_i$. As the non-singular case (see e.g. \cite[Section 4.1]{Dumas}), we can consider the corresponding operation in the universal cover of the regular set of $S$. It is not hard to see that the obtained surface admits a complex projective structure with prescribed cone singularities.

 For non-singular case, Thurston has shown that grafting along weighted simple closed curves extends continuously to arbitrary measured laminations. Note again that the distance from the support of any measured lamination to the cone points is bounded away from 0. Under a limit process, we can also consider the grafting along a measured lamination $\lambda\in\cML_{\mathfrak{p}}$ as the limit of the obtained complex projective structure under the grafting operation along $\cup_i t_i\gamma_i$ with $\cup_i t_i\gamma_i\rightarrow \lambda$ in $\cML_{\mathfrak{p}}$ (note that this is independent of the choice of $\cup_i t_i\gamma_i$).

\begin{definition}
Let $Gr_\theta:\cT_{\Sigma,\theta}\times \cML_\mathfrak{p}\to \cCP_\theta$ be the map associates to $(h,\lambda)$ the complex projective structure obtained by the above grafting operation on a hyperbolic surface $(\Sigma,h)$ along $\lambda$. We call it the grafting map.
\end{definition}

\begin{lemma}\label{lm:grafting}
$Gr_\theta=f_1\circ f$.
\end{lemma}

\begin{proof}
It suffices to show that for each hyperbolic end $M\in\cHE_{\mathfrak{p}}$, the complex projective structure induced on $\partial_{\infty}M$ can be obtained as the image of the pair $(h,\lambda)$ under the grafting map $Gr_\theta$, where $h$ and $\lambda$ are the induced hyperbolic metric and the bending lamination on $\partial_0M$, respectively. Indeed, we only need to prove this for the case that $\lambda$ is a simple closed geodesic $\gamma$ with the weight $\alpha>0$ which records the bending angle at $\gamma$.

Let $S=\partial_0M$ and consider the normal exponential map $ \exp: N^1S\times(0,+\infty)\rightarrow M$ defined in Lemma \ref{lm:geometry at infinity}. For each $r>0$, the subset $\exp(N^1(S\setminus\gamma)\times\{r\})$  of the equidistant surface $S_r$ at distance $r$ from $S$ has induced metric $I_r=\cosh^2(r)h$ for all $x\in S\setminus\lambda$. Moreover, the image $\exp(N^1(\gamma)\times\{r\})$ is an annulus $A_r$ embedded in $S_r$. By computation, $A_r$ has two boundary components of length $a_r=\cosh(r)\ell_{\gamma}(h)$ and the shortest distance between these two boundary components is $b_r=\sinh(r)\alpha$.

Therefore, for $x\in S\setminus\lambda$, the induced metric $I_r$ of the set $\exp(N^1(S\setminus\gamma)\times\{r\})$ satisfies that $e^{-2r}I_r\rightarrow h$ as $r\rightarrow+\infty$. On the other hand, the ration  (or module) of $A_r$, as $r\rightarrow +\infty$, satisfies that
\begin{equation*}
\frac{a_r}{b_r}=\frac{\cosh(r)}{\sinh(r)}\frac{\ell_{\gamma}(h)}{\alpha}\rightarrow \frac{\ell_{\gamma}(h)}{\alpha}=\Mod(A_{\gamma}),
\end{equation*}
where $A_{\gamma}=\gamma\times[0,\alpha]$ is the annulus replacing $\gamma$ in the grafting operation and $\Mod(A_{\gamma})$ is the module of $A_{\gamma}$. Therefore, the complex projective structure on $\partial_{\infty}M$ is $Gr_{\theta}(h,\lambda)$.
\end{proof}

\begin{proof}[Proof of Theorem \ref{tm:grafting}]
  This follows from Proposition \ref{prop: homeomorphism} and Lemma \ref{lm:grafting}.
  \end{proof}

%% file: foliation4.tex
\section{De Sitter spacetimes with particles and complex projective structures with cone singularities}
\label{sc:duality}

In this section, we consider the ``dual'' manifolds of hyperbolic ends with particles, that is, future-complete convex GHM de Sitter spacetimes with particles (see Definition \ref{def:ds spacetimes}). We describe this dual relation in terms of the complex projective structures induced on the boundary at infinity of either of these two dual manifolds.

It is interesting to ask whether every future-complete GHM de Sitter spacetime contains a strictly future-convex spacelike surface. This relates closely to a question posed in \cite[Section 6]{KS} whether every future-complete GHM de Sitter spacetime with particles contains a constant mean curvature spacelike surface, and a question asked in \cite{BSeppi} whether every future-complete GHM flat spacetime with particles contains a uniformly future-convex spacelike surface.

\subsection{The complex projective structure at infinity of a de Sitter spacetime $M^d\in\mathcal{DS}_{\theta}$}

Recall that every de Sitter spacetime in $\mathcal{DS}_{\theta}$ is future-complete. We denote by $\partial_{\infty}M^d$ the boundary at infinity of a de Sitter spacetime $M^d\in\mathcal{DS}_{\theta}$ and will show that $\partial_{\infty}M^d$ admits a complex projective structure with cone singularities of the same angles as the particles.

\textbf{The model space $W_{\alpha}$}.  Let $\alpha>0$ and let $\Gamma_0$ be a fixed, future-oriented complete timelike geodesic in $\mathbb{DS}_{3}$. Denote by $U$ the universal cover of the complement of $\Gamma_0$ in $\mathbb{DS}_{3}$ and denote by $W$ the completion of $U$, such that $W\setminus U$ is canonically identified to $\Gamma_0$, which is called the \emph{singular set of $W$}. We define $W_{\alpha}$ as the quotient of $W$ by the rotation of angle $\alpha$ around $\Gamma_0$. The image of the singular set of $W$ under this quotient is called the \emph{singular set of $W_{\alpha}$}.

Let $M^d$ be a future-complete convex GHM de Sitter spacetime with particles. It is clear that each singular point $x$ of $M^d$ has a neighborhood isometric to a subset of $W_{\alpha}$ with $\alpha$ equal to the total angle around the singular curve through $x$. Now we describe the geometry property of $M^d$ near the endpoints at infinity of the singular curves in $M^d$ by using the model $W_{\alpha}$, see the following lemma. Since $M^d$ contains a strictly future-convex spacelike surface, with an alternative version of Lemma \ref{lm:exp} for the de Sitter case with particles, the argument for the hyperbolic case with particles is adapted to the de Sitter case.

\begin{lemma}\label{lm:geometry at infinity for dS}
For each point $p_i\in\partial_{\infty}M^d$ which is the endpoint at infinity of a singular curve in $M^d$, $p_i$ has a neighborhood $U_i$ in $M^d$ isometric to a neighborhood of the endpoint at infinity of $\Gamma_0$ in $W_{\theta_i}$ which lies on $\mathbb{S}^2_{+}$, where $\theta_i$ is the total angle around that singular curve.
\end{lemma}

\begin{proof}
Now we prove the lemma in the following four steps:

\textbf{Step 1 :} Let $S^d\subset M^d$ be a strictly future-convex spacelike surface and let $NS^d$ be the space of future-directed vectors normal to $S^d$ (note that at a singular point $x\in S^d$, the``normal" vector is directed along the singular curve through $x$). Given $v=(x,n)\in NS^d$, we denote by $\exp(v)\in M^d$ the point $\gamma(1)$, where $\gamma:[0,1]\to M^d$ is the geodesic such that $\gamma(0)=x$ and $\gamma'(0)=n$, if it exists. This defines a map $\exp$ from a subset of $NS^d$ to $M^d$.

\textbf{Step 2 :} We claim that the map $\exp: NS^d\to M^d$ is well-defined on $NS^d$ and it is a homeomorphism onto its image. Note that $S^d$ is a Cauchy surface in $M^d$ and every geodesic starting in the direction of $NS^d$ is timelike, then there is no geodesic segment in the future of $S^d$ connecting two points of $S^d$ in the directions of $NS^d$. Applying an analogous argument used in Lemma \ref{lm:exp} for hyperbolic case, we have the claim.

\textbf{Step 3 :} The exponential map $\exp_{\infty}:NS^d\rightarrow \partial_{\infty} M^d$ is a homeomorphism, where $\exp_{\infty}$ is defined as the equivalence class of the geodesic ray which is the fiber of $NS^d$ over $x\in S^d$. This follows directly from Step 2.

\textbf{Step 4 :} By Step 3, for each point $p_i\in\partial_{\infty}M^d$ which is the endpoint at infinity of a singular curve in $M^d$, the singular curve is unique and we denote it by $l_i$. Assume that this singular curve $l_i$ intersects $S^d$ at $x_i$. Let $F_i$ be the fiber of $NS^d$ over $x_i$ and let $H_i$ be a small neighborhood of $F_i$ in $NS^d$. Consider $U_i=\exp(H_i)$. Note that $M^d$ is locally modelled on $W_{\theta_i}$ near the singular curve $l_i$ (where $l_i$ is identified as the singular set $\Gamma_0$ in $W_{\theta_i}$ and $\theta_i$ is the total angle around $l_i$). By the definition of de Sitter metrics with particles, $U_i$ contains a cylinder of exponentially expanding radius around $l_i$ along the future-direction. This implies the desired result.
\end{proof}

Note that the regular set $M^d_r$ of $M^d$ has a $(PSL_2(\mathbb{C}), \mathbb{DS}_3)$-structure and it is future-complete, we can define the boundary at infinity of $M^d_r$, denoted by $\partial_{\infty}M_r^d$, as for the hyperbolic case in Proposition \ref{CP1 sturctures at boundary}. Moreover, $\partial_{\infty}M_r^d$ carries a canonical complex projective structure. Combined with the geometric property of $M^d$ near the endpoints at infinity of the singular curves, as presented in Lemma \ref{lm:geometry at infinity for dS}, we have the following proposition.

\begin{proposition}\label{prop: dS to complex}
Let $M^d\in \cDS_\theta$ be a future-complete convex GHM de Sitter spacetime with particles. Then the boundary at infinity $\partial_\infty M^d$ is endowed with a complex projective structure with cone singularities of angle $\theta_i$ at the $p_i$.
\end{proposition}

\subsection{The construction of de Sitter spacetimes in $\mathcal{DS}_{\theta}$ from complex projective structures in $\mathcal{CP}_{\theta}$ }

To construct a convex GHM de Sitter spacetime with particles from a complex projective structure with cone singularities, we give the following result which ensures the existence and the uniqueness (up to isometry) of the maximal extension of a convex GH de Sitter spacetime with particles. This can be proved by adapting verbatim the argument given for the anti-de Sitter case in \cite[Proposition 6.24]{BBS11}.

\begin{proposition}\label{prop:maximal extension of dS}
            Let $M_0^d$ be a convex GH de Sitter spacetime with particles. Then there exists a unique (up to isometry) maximal extension of $M_0^d$, called $M^d$, in which $M_0^d$ can be isometrically embedded.
\end{proposition}

\begin{proposition}\label{prop:complex to dS}
Let  $\sigma\in \cCP_\theta$ be a complex projective structure with cone singularities. Then there is a unique future-complete convex GHM de Sitter spacetime with particles $M^d\in \cDS_\theta$, such that $\partial_{\infty}M^d$ is endowed with the complex projective structure $\sigma$.
\end{proposition}

\begin{proof}
By Lemma \ref{lm:cp to qd}, the Schwarzian derivative of the conformal map $id:(\Sigma_{\mathfrak{p}},\sigma)\rightarrow (\Sigma_{\mathfrak{p}},\sigma_F)$ is a meromorphic quadratic differential $q$ in $T_{c}^{*}\mathcal{T}_{\Sigma,\theta}$,  where $c$ is the common underlying conformal structure of $\sigma_F$ and $\sigma$.

Now we use $q$ to construct a future-complete convex GHM de Sitter spacetime $M^d$ with particles in the following two steps, as in Proposition \ref{prop:qd to he} for the hyperbolic case.

\textbf{Step 1 :} First we construct a future-complete GH de Sitter spacetime $M^d_0$ with the prescribed particles which is homeomorphic to $\Sigma\times\mathbb{R}_{\geq0}$.

As in the hyperbolic case (see the proof of Proposition \ref{prop:qd to he}), we use the same data at infinity. Let $I^{*}$ be a hyperbolic metric with the prescribed cone singularities in the conformal class $c$. Recall the notations that $\II_0^{*}=\Real \,q$, ${\II}^{*}=\frac{1}{2}I^{*}+{\II}^{*}_0$, $B^{*}=(I^{*})^{-1}{\II}^{*}$ and ${\III}^{*}=I^{*}(B^{*}\bullet,B^{*}\bullet)$.

Let $M^d_0$ be the set $\Sigma\times [t_0,+\infty)$ with the metric
\begin{equation*}
g^d_0=-dt^2+\frac{1}{2}(e^{2t}I^{*}-2{\II}^{*}+e^{-2t}{\III}^{*}),
\end{equation*}
where $t_0$ is to be determined. We claim that $M^d_0$ is a convex GH de Sitter spacetime with particles if we choose $t_0$ large enough. Denote $I^d_{t}=\frac{1}{2}(e^{2t}I^{*}-2\II^{*}+e^{-2t}{\III}^{*})$. Then we have
\begin{equation*}
\begin{split}
I^d_{t}&=\frac{1}{2}I^*((e^tE-e^{-t}B^{*})\bullet, (e^tE-e^{-t}B^{*})\bullet), \\
\II^d_t&=\frac{1}{2}\frac{dI^d_t}{dt}
=\frac{1}{2}I^{*}((e^tE-e^{-t}B^{*})\bullet,(e^tE+e^{-t}B^{*})\bullet).
\end{split}
\end{equation*}

Denote $B^d_t=(e^tE-e^{-t}B^{*})^{-1}(e^tE+e^{-t}B^{*})$. Similarly as the hyperbolic case (see Proposition \ref{prop:qd to he}), one can check that $(I^d_t, B^d_t)$ satisfies the following conditions:
\begin{itemize}
\item $B^d_t$ is self-adjoint for $I^d_t$: $I^d_t(B^d_t\bullet,\bullet)=I^d_t(\bullet,B^d_t\bullet)$. This follows from the fact that $\II^*$ is self-adjoint for $I^*$ (since $\II^*_0$ is the real part of a quadratic differential $q$).
  \item $(I^d_t,B^d_t)$ satisfies the Gauss equation for surfaces in $dS_3$: $K_{I^d_t}=1-\det B^d_t$, where $K_{I^d_t}$ is the sectional curvature of $I^d_t$. 

  \item $(I^d_t,B^d_t)$ satisfies the Codazzi equation: $d^{\nabla^{I^d_t}}B^d_t=0$, where $\nabla^{I^d_t}$ is the Levi-Civita connection of $I^d_t$.
 \item ($I^d_t$, $B^d_t$) satisfies the following equality:
 \begin{equation*}
 I^d_{t+s}=I^d_t((\cosh(s)E+\sinh(s)B^d_t)\bullet,  (\cosh(s)E+\sinh(s)B^d_t)\bullet),
 \end{equation*}
 for all $t, s>0$. This follows from a direct computation.

\end{itemize}

     Denote by $\lambda^{*}$, $\mu^{*}$ (resp. $\lambda^d_t$, $\mu^d_t$) the eigenvalues of $B^{*}$ (resp. $B^d_t$). By computation,
     \begin{equation*}
     \lambda^d_t=\frac{e^t+e^{-t} \lambda^{*}}{e^t-e^{-t} \lambda^{*}},
      \quad\quad
     \mu^d_t=\frac{e^t+e^{-t} \mu^{*}}{e^t-e^{-t} \mu^{*}}.
     \end{equation*}

If $t_0$ is large enough, the eigenvalues $\lambda^d_{t_0}$, $\mu^d_{t_0}$ of $(\Sigma\times\{t_0\}, I^d_{t_0})$ are both positive. Let the positive direction of $t$ be the future direction. Combined with the above properties of $(I^d_t, B^d_t)$, this shows that $M^d_{0}$ is a future-complete convex GH de Sitter spacetime with particles.

From the argument in Proposition \ref{prop:qd to he}, we have that $B^*$ tends to $\frac{1}{2}E$ at each cone singularity $p_i$. Therefore, $I^d_t$ tends to $\frac{1}{2}(e^t-\frac{1}{2}e^{-t})^2I^*$ at $p_i$. This shows that the total angle around the singular curve $\{p_i\}\times[t_0,+\infty)$ of $M^d_{0}$ is $\theta_i$.

\textbf{Step 2 :} We construct the desired de Sitter spacetime $M^d$ with particles via $M^d_0$.

Indeed, by Proposition \ref{prop:maximal extension of dS}, $M^d_0$ admits a unique maximal extension, say $M^d$. We will show that the induced complex projective structure $\sigma$ on $\partial_\infty M^d$ satisfies the required condition.

A direct computation shows that $I^{*}=\frac{1}{2}e^{-2t}{G^d_t}_{*}(I^d_t+2\II^d_t+\III^d_t)$, where $G^d_t$ is the Gauss map from $(\Sigma\times\{t\}, I^d_t)$ to $\partial_{\infty}M^d$. This implies that the conformal structure induced on $\partial_{\infty}M^d$ by the de Sitter metric on $M^d$ is $c$. Note that the expressions of the first, second and third fundamental forms of the surfaces $\Sigma_t$ in the foliation near the boundary at infinity of $M^d$ can be obtained by replacing the shape operator $B^*$ in Proposition \ref{prop:qd to he} by $-B^*$. An adaption of the argument for the hyperbolic case (see \cite[Lemma 8.3]{KS2}) shows that the real part of the Schwarzian derivative of the natural map $\phi:(\partial_{\infty} M^d,\sigma^d)\rightarrow (\partial_{\infty} M^d,\sigma^d_F)$ is $\II_0^{*}$, where $\sigma^d$ is the complex projective structure induced on $\partial_{\infty} M^d$ and $\sigma^d_F$ is the Fuchsian complex projective structure of $\sigma^d$. Hence, $\Real \mathcal{S}(\phi)=\II^{*}_0=\Real q$.
This implies that $\mathcal{S}(\phi)=q$. Note also that $\sigma^d_F=I^{*}=\sigma_F$, then $\sigma^d=\sigma$. This implies that the complex projective structure induced on $\partial_{\infty} M^d$ is exactly $\sigma$.
\end{proof}

For convenience, we give the commutative diagram in Figure \ref{fig:dS}, which shows the relations among the spaces related to $\mathcal{DS}_{\theta}$ and the following maps $g_1$, $g_2$, $g_3$ are all homeomorphisms (see e.g. Proposition \ref{prop: homeomorphism}).
\begin{itemize}
\item the map $g_1:\mathcal{DS}_{\theta}\to \mathcal{CP}_{\theta}$ sending a de Sitter spacetime with particles to the its complex projective structure at infinity, see Proposition \ref{prop: dS to complex},
\item the map $g_2$ defined to be the map $f_2$ in Proposition \ref{prop: homeomorphism},
\item the map $g_3:T^{*}\mathcal{T}_{\Sigma,\theta}\to \mathcal{DS}_{\theta}$ reconstructing a de Sitter spacetime with particles from the data of a hyperbolic 
metric and a traceless Codazzi tensor on the boundary at infinity, see Proposition \ref{prop:complex to dS}.
\end{itemize}

\emph{Proof of Theorem 1.4} This follows from Proposition \ref{prop: dS to complex} and Theorem \ref{prop:complex to dS}.

\begin{figure}
$\xymatrix{
& \mathcal{DS}_{\theta} \ar[dr]^{g_1}             \\
  T^*\mathcal{T}_{\Sigma,\theta}  \ar[ur]^{g_3} & &  \mathcal{CP}_{\theta}         \ar[ll]_{g_2}}$
  \caption{\small{A diagram showing the relations among the spaces related to $\mathcal{DS}_{\theta}$.}}
     \label{fig:dS}
\end{figure}
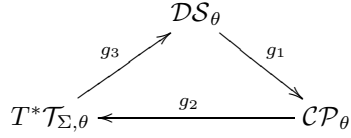

\subsection{The duality between $\cHE_{\theta}$ and $\cDS_{\theta}$}
\label{subsection: dual relation between hE and dS}

 Combining Theorem \ref{tm:projective} and Theorem \ref{tm:projective-ds}, we can define a natural map, say $\delta$, which is a homeomorphism from $\cHE_\theta$ to $\cDS_\theta$ sending a hyperbolic end with particles to the unique future-complete convex GHM de Sitter spacetime with the same complex projective structure at infinity.

Let $M\in \cHE_\theta$ be a non-degenerate hyperbolic end with particles, and let $M^d\in \cDS_\theta$ be the dual future-complete convex GHM de Sitter spacetime with particles.
We also describe a duality between closed strictly concave surfaces in $M$ and closed strictly future-convex surfaces $S^d$ in $M^d$.

Let $S\subset M$ be a closed, strictly concave surface. We define a dual surface $S^d\subset M^d$ of $S$, as the surface satisfying the following properties:
\begin{itemize}
  \item $S^d$ is a strictly future-convex spacelike surface in $M^d$.
  \item There is a unique diffeomorphism $u:S\to S^d$ such that $u^*I_d=\III$ and $u^*\III_d=I$, where $I,\III$ are the induced metric and third fundamental form on $S$, and $I_d$ and $\III_d$ are the induced metric and third fundamental form on $S^d$.
\end{itemize}

Conversely, given a closed, strictly future-convex spacelike surface $S^d\subset M^d$, we can also define a dual surface $S\subset M$ of $S^d$, in an analogous way as above. It remains to show that the definition of the duality for surfaces is well-defined. By observation, it suffices to show the existence and uniqueness of the dual surface $S^d\subset M^d$ of a closed, strictly concave surface $S$ in a hyperbolic end $M$ with particles defined above. Equivalently, it suffices to show Theorem \ref{tm:dual-surfaces}.\\

To show Theorem \ref{tm:dual-surfaces}, it is convenient to clarify the relation between hyperbolic ends with particles  (resp. convex GHM de Sitter spacetimes with particles)  and the data at infinity. We will then see in the next subsection that the same description applies in the de Sitter case.



\subsection{Hyperbolic ends with particles and the data at infinity}
\label{subsection for hyperbolic case}
\subsubsection{The data at infinity obtained from an equidistant foliation near the boundary at infinity of $M\in\mathcal{HE}_{\theta}$.}\label{subsubsection:1 for hyperbolic case}

Let $M$ be a hyperbolic end with particles and let $S$ be a strictly concave surface in $M$, with the induced metric $I$, the shape operator $B$, and the second fundamental form $\II$. Consider an equidistant foliation $(S_r)_{r>0}$, with $S_r$  the equidistant surface obtained at distance $r$ from $S$ along the orthogonal geodesics on the concave side of $S$. Define the data at infinity $(I^*, \II^*)$ as follows:
 \begin{equation}\label{eq:data at infinity for hE}
\begin{split}
I^*&=\frac{1}{2}e^{-2r}{G_r}_*(I_r+2\II_r+\III_r),\\
\II^*&=\frac{1}{2}e^{-2r}{G_r}_*(I_r-\III_r),
\end{split}
\end{equation}
where $I_r$, $\II_r$, $\III_r$  are respectively the induced metric, second and third fundamental forms on $S_r$ in $M$, while $G_r$ is the Gauss map from $S_r$ to the boundary at infinity $\partial_{\infty}M$ of $M$. One can check by direct computation that the data $(I^*, \II^*)$ defined above is independent of $r$. 

It is not hard to check that (see e.g. \cite[Remark 5.4 and Remark 5.5]{KS2}) the data $(I^*, \II^*)$ satisfies the Codazzi equation and a modified version of the Gauss equation for surfaces embedded in ${\mathbb{H}}^3$:
 \begin{equation}\label{eq:Codazzi and Gauss for hyperbolic ends}
\begin{split}
d^{{\nabla}^{I^*}}\II^*&=0, \\
\trace_{I^*}\II^*&=-K_{I^*}~,
\end{split}
\end{equation}
where $\nabla^{I^*}$ is the Levi-Civita connection of $I^*$ and $K_{I^*}$ is the Gauss curvature of $I^*$.

Conversely,  $I_r$, $\II_r$, and the shape operator $B_r$ of $S_r$ can be rewritten by using the data at infinity $(I^*, \II^*)$ in the following way (see \cite[Lemma 5.6]{KS2}).
\begin{equation}\label{eq:equidistant and infinity}
\begin{split}
I_{r}&=\frac{1}{2}I^*((e^rE+e^{-r}B^{*})\bullet, (e^rE+e^{-r}B^{*})\bullet), \\
\II_r&=\frac{1}{2}I^{*}((e^rE+e^{-r}B^{*})
\bullet,(e^rE-e^{-r}B^{*})\bullet),\\
B_r&=(e^rE+e^{-r}B^{*})^{-1}(e^rE-e^{-r}B^{*}),
\end{split}
\end{equation}
where $B^*=(I^*)^{-1}\II^*$.

\subsubsection{The hyperbolic end with particles determined by a particular couple on $\Sigma$. }
\label{subsubsection:2 for hyperbolic case}

Let $(I'^*, \II'^*)$ be a couple with $I'^*$ a Riemannian metric on $\Sigma$, and $\II'^*$ a bilinear symmetric form on $T\Sigma$ (defined outside the singular locus) satisfying the following conditions, called \emph{Condition $(\star)$} for convenience.
\begin{itemize}
  \item $(I'^*, \II'^*)$ assumes the two equations in \eqref{eq:Codazzi and Gauss for hyperbolic ends} by replacing $(I^*,\II^*)$ with $(I'^*,\II'^*)$.
  \item The determinant of $\II'^*$ with respect to $I'^*$ remains bounded.
\end{itemize}

In particular, the previous data at infinity $(I^*, \II^*)$ obtained from $(S_r)_{r>0}$ in Section \ref{subsubsection:1 for hyperbolic case} satisfies Condition $(\star)$.
Denote $B'^*=(I'^{*})^{-1}\II'^*$. Consider the manifold $\Sigma\times[0,+\infty)$ with the following metric
\begin{equation*}
g_0=dr^2+I'_r,
\end{equation*}
where $I'_r$ is defined as the formula for $I_r$ in \eqref{eq:equidistant and infinity} by replacing $(I^*,\II^*, B^*)$ with $(I'^*,\II'^*, B'^*)$. By Condition $(\star)$, it can be checked as Step 1 in the proof of Proposition \ref{prop:qd to he} that $(I'_r, B'_r)$ determines a hyperbolic end with particles, denoted by $M'$, with $(\Sigma\times\{r\})_{r>0}$ an equidistant foliation near the boundary at infinity of $M'$. Moreover, the data at infinity obtained from  $(\Sigma\times\{r\})_{r>0}$ as in Section \ref{subsubsection:1 for hyperbolic case} is exactly the given couple $(I'^*,\II'^*)$. This shows that the prescribed couple $(I'^*,\II'^*)$ completely determines a hyperbolic end with particles.

To verify Theorem \ref{tm:dual-surfaces}, we also need the following proposition, which follows from a particular case (i.e. the case of 2+1 dimensional Poincar\'{e}-Einstein manifold) of Theorem 1.5 in \cite{Schlenker1}.

\begin{proposition}\label{prop:hyperbolic ends and infinite data}
Let $(I_1^*,\II_1^*)$ and $(I_2^*,\II_2^*)$ be two couples satisfying Condition $(\star)$. Then $(I_1^*,\II_1^*)$ and $(I_2^*,\II_2^*)$ characterize the same hyperbolic end with particles if and only if they satisfy the following relation:
\begin{equation}\label{eq:relation}
\begin{split}
I_2^*&=e^{2u}I_1^*,\\
\II_2^*&=\II_1^*+\Hess(u)-du\otimes du+\frac{1}{2}||du||^2_{I_1^*}I_1^*,
\end{split}
\end{equation}
where $u$ is a continuous function on $\Sigma$ and $\mathcal{C}^2$ function on $\Sigma_{\mathfrak{p}}$. Moreover, $\mathcal{HE}_{\theta}$ is parameterized by the space of the couples satisfying Condition $(\star)$ identified by the relation \eqref{eq:relation}.
\end{proposition}

\subsection{De Sitter spacetimes with particles and the data at infinity}
\label{subsection for dS case}

\subsubsection{The data at infinity obtained from an equidistant foliation near the boundary at infinity of $M^d\in\mathcal{DS}_{\theta}$.}
\label{subsubsection:1 for dS case}

Similarly, we can define the data at infinity, called $(I^{d*}, \II^{d*})$, of a future-complete convex GHM de Sitter spacetime with particles $M^d$ by an equidistant foliation $(S^d_{t})_{t>0}$, where $S^d_t$  is the equidistant surface obtained at distance $t$ from $S^d$ along the orthogonal geodesics on the convex side of a strictly future-convex spacelike surface $S$ in $M^d$. Define the data at infinity $(I^{d*}, \II^{d*})$ as follows:
 \begin{equation}\label{eq:data at infinity for dS}
\begin{split}
I^{d*}&=\frac{1}{2}e^{-2t}{G^d_t}_*(I^d_t+2\II^d_t+\III^d_t),\\
\II^{d*}&=\frac{1}{2}e^{-2t}{G^d_t}_*(\III^d_t-I^d_t),
\end{split}
\end{equation}
here $I^d_t$, $\II^d_t$, $\III^d_t$  are respectively the induced metric, second and third fundamental forms on $S^d_t$ in $M^d$, while $G^d_t$ is the Gauss map from $S^d_t$ to the boundary at infinity $\partial_{\infty}M^d$ of $M^d$. One can check by direct computation that the data $(I^{d*}, \II^{d*})$ defined above is independent of $t$.

It is not hard to check that the data $(I^{d*}, \II^{d*})$ satisfies the Codazzi equation and a modified version of the Gauss equation for surfaces embedded in $dS_3$ (indeed, these equations turn out to be the same as those in \eqref{eq:Codazzi and Gauss for hyperbolic ends} for the hyperbolic case):
 \begin{equation}\label{eq:Codazzi and Gauss}
\begin{split}
d^{{\nabla}^{I^{d*}}}\II^{d*}&=0, \\
\trace_{I^{d*}}\II^{d*}&=-K_{I^{d*}}~,
\end{split}
\end{equation}
where $\nabla^{I^{d*}}$ is the Levi-Civita connection of $I^{d*}$ and $K_{I^{d*}}$ is the Gauss curvature of $I^{d*}$.

Conversely,  $I^d_t$, $\II^d_t$, and the shape operator $B^d_t$ of $S^d_t$ can be rewritten by using the data at infinity $(I^{d*}, \II^{d*})$ in the following way (one can check this by direct computation).
\begin{equation}\label{eq:equidistant and infinity for dS}
\begin{split}
I^d_{t}&=\frac{1}{2}I^{d*}((e^tE-e^{-t}B^{d*})\bullet, (e^tE-e^{-t}B^{d*})\bullet), \\
\II^d_t&=\frac{1}{2}I^{d*}((e^tE-e^{-t}B^{d*})
\bullet,(e^tE+e^{-t}B^{d*})\bullet),\\
B^d_t&=(e^tE-e^{-t}B^{d*})^{-1}(e^tE+e^{-t}B^{d*}),
\end{split}
\end{equation}
where $B^{d*}=(I^{d*})^{-1}\II^{d*}$.

\subsubsection{The de Sitter spacetime with particles determined by  a particular couple on $\Sigma$. }
\label{subsubsection:2 for dS case}

Let $(I'^*, \II'^*)$ be a couple satisfying Condition $(\star)$. In particular, the previous data at infinity $(I^{d*}, \II^{d*})$ obtained from $(S^d_t)_{t>0}$ in Section \ref{subsubsection:1 for dS case} satisfies Condition $(\star)$.

Denote $B'^*=(I'*)^{-1}\II'^*$. Consider the manifold $\Sigma\times[0,+\infty)$ with the following metric
\begin{equation*}
g^d_0=-dt^2+I'_t,
\end{equation*}
where $I'_t$ is defined as the formula for $I^d_t$ in \eqref{eq:equidistant and infinity for dS} by replacing $(I^{d*},\II^{d*}, B^{d*})$ with $(I'^*,\II'^*, B'^*)$. By Condition $(\star)$, it can be checked as Step 1 in the proof of Proposition \ref{prop:complex to dS} that $(I'_t, B'_t)$ determines a future-complete convex GHM de Sitter spacetime with particles, denoted by $M'^d$, with $(\Sigma\times\{t\})_{t>0}$ an equidistant foliation near the boundary at infinity of $M'^d$. Moreover, the data at infinity obtained from  $(\Sigma\times\{t\})_{t>0}$ as in Section \ref{subsubsection:1 for dS case} is exactly the given couple $(I'^*,\II'^*)$. This shows that the prescribed couple $(I'^*,\II'^*)$ completely determines a future-complete convex GHM de Sitter spacetime with particles.

As a consequence, we have a result for the de Sitter case analogous to Proposition \ref{prop:hyperbolic ends and infinite data} for the hyperbolic case.

\begin{proposition}\label{prop:dS and infinite data}
Let $(I_1^*,\II_1^*)$ and $(I_2^*,\II_2^*)$ be two couples satisfying Condition $(\star)$. Then $(I_1^*,\II_1^*)$ and $(I_2^*,\II_2^*)$ characterize the same future-complete convex GHM de Sitter spacetime with particles if and only if they satisfy the relation \eqref{eq:relation}. Moreover,
$\mathcal{DS}_{\theta}$ is parameterized by the space of the couples satisfying Condition $(\star)$ identified by the relation \eqref{eq:relation}.
\end{proposition}

\emph{Proof of Theorem \ref{tm:dual-surfaces}} : By the definition of the dual relation between $M$ and $M^d$ (see Section \ref{subsection: dual relation between hE and dS}) and combining Proposition \ref{prop:hyperbolic ends and infinite data} and Proposition \ref{prop:dS and infinite data}, $M$ and $M^d$ are indeed parameterized by the same data at infinity, denoted by $(I'^*, \II'^*)$, which is obtained from the same complex projective structure with cone singularities induced at infinity of $M$ and $M^d$ (see e.g. Proposition \ref{prop:qd to he} and Proposition \ref{prop:complex to dS}).

Note that from the given embedded strictly concave surface $S\subset M$ we can construct an equidistant foliation $(S_r)_{r>0}$ near $\partial_{\infty}M$. Hence, $M$ is also characterized by the couple $(I^*,\II^*)$, which is the data at infinity obtained from the foliation $(S_r)_{r>0}$, as shown in Section \ref{subsubsection:1 for hyperbolic case}. Proposition \ref{prop:hyperbolic ends and infinite data} implies that $(I^*,\II^*)$ and $(I'^*, \II'^*)$ satisfy the relation \eqref{eq:relation}.

Now we construct a future-complete convex GHM de Sitter spacetime with particles, called $M_1^d$, by using an adapted embedding data, denoted by $(I^d, B^d)$, obtained from $(S, I, B)$, where $B$ is the shape operator of $S$ in $M$, $(I^d, B^d)$ is defined as follows:
\begin{center}
$I^d:=\III$, \quad $B^d:=B^{-1}$.
\end{center}

It is not difficult to check that $(I^d, B^d)$ satisfies the Codazzi-Gauss equations for surfaces embedded in $dS_3$ (this follows from a computation using Proposition \ref{computation of connection and curvature} and the fact that $(I, B)$ satisfies the Codazzi-Gauss equations for surfaces embedded in ${\mathbb{H}}^3$). Moreover, $B^d$ is self-disjoint for $I^d$ with positive eigenvalues. 
Now we consider the manifold $\Sigma\times [0,+\infty)$, called $M_0^d$, with the following metric:
        \begin{equation*}
            g^d_{0}=-dt^{2}+I^d((\cosh(t)E+\sinh(t)B^d)\bullet,(\cosh(t)E+\sinh(t)B^d)\bullet)~,
        \end{equation*}
        where $E$ is the identity isomorphism on $T\Sigma$ and $t\in[0,+\infty)$. Combined with the above properties of $(I^d, B^d)$, it follows that $M_0^d$ is a future-complete convex GH dS spacetime with particles. Let $M^d_1$ be the (unique) maximal extension of $M_0^d$ (this is ensured by Proposition \ref{prop:maximal extension of dS}). Moreover, $\Sigma\times\{t\}$, called $S^d_t$, is the equidistant surface in $M^d_1$ at a distance $t$ on the convex side from the strictly future-convex surface $\Sigma\times\{0\}$, called $S^d$, with the induced metric $I^d$ and shape operator $B^d$.

        Therefore, $M^d_1$ has the data at infinity, called $(I^{d*}, \II^{d*})$, which is obtained from the foliation $(S^d_t)_{t>0}$ near $\partial_{\infty}M_1^d$. One can check by using the formulas \eqref{eq:data at infinity for hE} and \eqref{eq:data at infinity for dS} that $(I^{d*}, \II^{d*})=(I^{*}, \II^{*})$. Therefore, $(I^{d*}, \II^{d*})$ and $(I'^*, \II'^*)$ satisfy the relation \eqref{eq:relation}. Using Proposition \ref{prop:dS and infinite data} again, the manifold $M_1^d$ characterized by $(I^{d*}, \II^{d*})$ is the same as the manifold $M^d$ characterized by $(I'^{*}, \II'^{*})$.

        Therefore, $S^d$ is a strictly future-convex spacelike surface in $M^d$. Since the boundary at infinity $\partial_{\infty}M$ (resp. $\partial_{\infty}M^d$) of $M$ (resp. $M^d$) can be identified as a complex projective surface with prescribed cone singularities. There is a natural correspondence between the points on $\partial_{\infty}M$ and the points on $\partial_{\infty}M^d$ through the Gauss normal flow starting from $S\subset M$ (resp. $S^d\subset M^d$). Let $u:=({G^d})^{-1}\circ G$, where $G^d$ (resp. $G$) is the Gauss map from $S^d$ (resp. $S$) to $\partial_{\infty}M^d$ (resp. $\partial_{\infty}M$). Then $u: S\to S^d$ is a diffeomorphism (outside the singular locus) such that $u^*I^d=\III$ and $u^*\III^d=I$. This completes the proof of Theorem \ref{tm:dual-surfaces}.\\

\emph{Proof of Proposition \ref{pr:dual}} :
Denote by $K$ the Gauss curvature of a strictly concave surface $S\subset M$ in Theorem \ref{tm:dual-surfaces} and denote by $K^d$ the Gauss curvature of the dual strictly future-convex surface $S^d\subset M^d$. It follows from the argument of  Theorem \ref{tm:dual-surfaces} that $K^d$ is equal to the Gauss curvature of the third fundamental form on $S$,
that is, $K^d=K/(K+1)$. Conversely, $K$ is equal to the Gauss curvature of the third fundamental form on $S^d$, that is,  $K=K^d/(1-K^d)$.
Therefore, $K$ is a constant in $(-1,0)$ if and only if $K^d$ is a constant in $(-\infty,0)$, related by an equality $K^d=K/(K+1)$. This shows Proposition \ref{pr:dual}.

%% file: foliation5.tex
\section{Parametrization of $\mathcal{HE}_{\theta}$ by $\mathcal{T}_{\Sigma,\theta}\times\mathcal{T}_{\Sigma,\theta}$ in terms of constant curvature surfaces}

In this section, we will prove Theorem \ref{tm:foliation} by parameterizing $\mathcal{HE}_{\theta}$ in terms of constant curvature surfaces. We 
consider hyperbolic manifolds with particles homeomorphic to $\Sigma\times \mathbb{R}_{>0}$, with a metric boundary orthogonal to the singular locus. Moreover, the surfaces we consider in a hyperbolic manifold with particles are assumed to be embedded closed surfaces isotopic 
to $\Sigma\times \{t\}$ for some $t>0$ and orthogonal to the singular curves. In order to define the parameterization map, we first give the following lemma.


    \subsection{The definition of the map $\phi_{K}$}

    \begin{lemma}
        \label{definition of the map}
            Let $K\in(-1, 0)$ and let $(h,h')\in\mathfrak{M}^{\theta}_{-1}\times\mathfrak{M}^{\theta}_{-1}$ be a pair of normalized metrics. Then there exists a unique hyperbolic end $M$ with particles which contains a surface of constant curvature $K$, with the induced metric $I=(1/|K|)h$ and the third fundamental form $\III=(1/|K^{*}|)h'$, where $K^{*}=K/(1+K)$.
    \end{lemma}

    \begin{proof}
        Let $b:T\Sigma\rightarrow T\Sigma$ be the bundle morphism associated to $h$ and $h'$ by Definition \ref{defintion of normalized representatives}, so that $h'=h(b\bullet, b\bullet)$. Let $I=(1/|K|)h$. We equip $\Sigma$ with the metric $I$ and consider a bundle morphism $B:T\Sigma\rightarrow T\Sigma$, which is defined by $B=\sqrt{1+K}b$. By the properties of $h$ and $b$, it follows that
        \begin{itemize}
            \item $(\Sigma, I)$ has constant curvature $K$.
            \item $B$ is self-adjoint for $I$ with positive eigenvalues.
            \item $B$ satisfies the Codazzi equation: $d^{\nabla^{I}}B=0$, where $\nabla^{I}$ is the Levi-Civita connection of $I$.
            \item $B$ satisfies the Gauss equation: $K=-1+\det(B)$.
        \end{itemize}

        Consider the manifold $\Sigma\times (-\varepsilon,+\infty)$ with the following metric (here $\varepsilon>0$ is a sufficiently small number):
        \begin{equation*}
            g_{0}=dt^{2}+I((\cosh(t)E+\sinh(t)B)\bullet,(\cosh(t)E+\sinh(t)B)\bullet)~,
        \end{equation*}
        where $E$ is the identity isomorphism on $T\Sigma$ and $t\in(-\varepsilon,+\infty)$. One can check that $\Sigma\times (-\varepsilon,+\infty)$ endowed with the metric $g_0$ is a hyperbolic manifold with particles, denoted by $M_0$, which has a concave metric boundary (note that $B$ has positive eigenvalues, then $\Sigma\times\{0\}$ with the induced metric is strictly concave and we can construct such a manifold by taking $\varepsilon$ small enough), and each line $\{p_{i}\}\times (-\varepsilon,+\infty)$ corresponds to a singular curve, around which the total angle is $\theta_i$. Furthermore, for each $t\in (-\varepsilon,+\infty)$, the surface $\Sigma\times \{t\}$ is the equidistant surface at an oriented distance $t$ from $\Sigma\times\{0\}$, where $t>0$ corresponds to 
        the concave side of $\Sigma\times\{0\}$.

        By Proposition \ref{existence and uniqueness of maximal extension}, there exists a unique maximal concave extension $M$ of $M_0$, which is a hyperbolic end with particles, such that the metric on $M$ restricted to the subset $\Sigma\times (-\varepsilon,+\infty)$ is exactly $g_{0}$. In particular, $M$ contains a concave surface of constant curvature $K$ (which is orthogonal to the singular curves) at $\Sigma\times\{0\}$, with the induced metric $I=(1/|K|)h$ and the third fundamental form
        \begin{equation*}
            \III=I(B\bullet,B\bullet)
            =\frac{1}{|K|}h(\sqrt{1+K}b\bullet,\sqrt{1+K}b\bullet)
            =\frac{1}{|K^{*}|}h',
        \end{equation*}
        where $|K^{*}|=K/(1+K)$. This shows the existence of the required manifold $M$. The uniqueness follows directly from the constrain conditions of the hyperbolic end with particles.
    \end{proof}

    It can be checked as Lemma 3.3 in  \cite{CS} that for any $(\tau,\tau')\in\mathcal{T}_{\Sigma,\theta}\times\mathcal{T}_{\Sigma,\theta}$, if $(h,h')$ and $(h_1,h'_1)$ are two normalized representatives of $(\tau,\tau')$, then the hyperbolic end with particles associated to $(h,h')$ and $(h_1,h'_1)$, as described in Lemma \ref{definition of the map}, are isotopic. Now we are ready to give the definition of the parametrization map $\phi_K$.

    \begin{definition}
    For any $K\in(-1,0)$, define the map $\phi_K:\mathcal{T}_{\Sigma, \theta}\times\mathcal{T}_{\Sigma,\theta}\rightarrow \mathcal{HE}_{\theta}$ by assigning to an element $(\tau,\tau')\in\mathcal{T}_{\Sigma, \theta}\times\mathcal{T}_{\Sigma,\theta}$ the isotopy class of the hyperbolic end with particles satisfying the property prescribed in Lemma \ref{definition of the map}.
    \end{definition}

    We show that the map $\phi_K$ is a homeomorphism, as stated in the following proposition.

      \begin{proposition}\label{prop:parametrization map}
    For any $K\in(-1,0)$ and $\theta=(\theta_1,...,\theta_{n_0})\in(0,\pi)^{n_0}$, the map $\phi_K: \mathcal{T}_{\Sigma,\theta}\times\mathcal{T}_{\Sigma,\theta}\rightarrow \mathcal{HE}_{\theta}$ is a homeomorphism.
    \end{proposition}

The proof will be given below, after some preliminary lemmas.

        \subsection{The injectivity of the map $\phi_K$}

            We prove this property by applying the Maximum Principle outside the singular locus and a specialized analysis near cone singularities. The idea is similar to that given in \cite[Section 3.2]{CS} for the case of AdS manifold with particles. Indeed, this argument is applicable to two concave surfaces which behave ``umbilically" (i.e. the limits of the principal curvatures tend to be the same) at singular points and satisfy the property that the supremum of the Gauss curvature over all the points of one surface is less than the infimum of those of the other surface (see Lemma \ref{lm:injective} for more details).

            Let $M\in\mathcal{HE}_{\theta}$ be a hyperbolic end with particles. Let $S\subset M$ be a concave surface of constant curvature $K\in(-1,0)$. Consider the minimal Lagrangian map (see Corollary \ref{hyperbolic metrics and bundle morphism}) associated to two hyperbolic metrics $|K|I, |K^{*}|\III \in \mathfrak{M}^{\theta}_{-1}$,  where $|K^{*}|=K/(1+K)$, and $I$ (resp. $\III$)  is the first (resp. third) fundamental form of $S$.  By the last statement of Proposition \ref{Toulisse2},  both principal curvatures on $S$ tend to $k=\sqrt{1+K}$ at the intersection $p_i$ with the singular curve $l_{i}$ in $M$ for $i=1,...,n_0$.

    The following theorem is an alternative version of the Maximum Principle Theorem (see e.g. \cite[Lemma 2.3]{BBZ1}, \cite[Theorem 3.10]{CS}) for the case of hyperbolic ends with particles.

    \begin{theorem}[Maximum Principle]
        \label{Maximum principle for regular points}
            Let $M$ be a hyperbolic end with particles. Let $S$ and $S'$ be two concave surfaces in $M$. Assume that $S$ and $S'$ intersect at a regular point $x$, and assume that $S'$ is contained on the concave side of $S$ in $M$. Then the product of the principal curvatures of $S'$ at $x$ is smaller than or equal to that of $S$.
    \end{theorem}

    To show the injectivity of  $\phi_K$, we first 
    state the following two lemmas, which follow from a direct computation.

    \begin{lemma}
        \label{the property of the pushing surface}
            Let $M$ be a hyperbolic end with particles and let $S$ be a concave surface in $M$. Consider a map $\psi^{t}:S\rightarrow M$ defined by $\psi^{t}(x)=\exp_{x}(t\cdot n_x)$, where $n_x$ is the $\partial_{\infty }M$-directed unit normal vector at $x$ of $S$ in $M$. Then for each regular point $x\in S$, we have

            \begin{enumerate}[(1)]
                \item $\psi^{t}$ is an embedding in a neighbourhood of $x$ for all $t>0$.
                \item  The principal curvatures of $\psi^{t}(S)$ at the point $\psi^{t}(x)$ are given by
                    \begin{equation*}
                        \lambda^{t}(\psi^{t}(x))
                        =\frac{\lambda(x)+\tanh (t)}{1+\lambda(x)\tanh (t)},\qquad
                        \mu^{t}(\psi^{t}(x))=\frac{\mu(x)+\tanh (t)}{1+\mu(x)\tanh (t)},
                    \end{equation*}
                    where $\lambda(x)$ and $\mu(x)$ are the principal curvatures of $S$ at $x$.
                \item Fix $x\in S$, if $\lambda(x)\mu(x)\in(0,1)$, then $F(t)=\lambda^{t}(\psi^{t}(x))\cdot\mu^{t}(\psi^{t}(x))$ is strictly increasing in $(0,+\infty)$.
            \end{enumerate}
    \end{lemma}

    \begin{lemma}
         \label{comparison between principal curvatures at singularities}
            Let $M$ be a hyperbolic end with particles. Let $S$, $S'$ be two  concave surfaces in $M$. Assume that $S$ and $S'$ intersect at a singular point $x$ such that the limits of both principal curvatures of $S$ at $x$ are equal to $k>0$, and the limits of both principal curvatures of $S'$ at $x$ are equal to $k'>0$. If there exists a neighbourhood $U$ of $x$ in $S$ and a neighbourhood $U'$ of $x$ in $S'$ such that $U'$ is on the concave side of $U$, then $k'\leq k$.
    \end{lemma}

    Let $S$ be a concave surface in a hyperbolic end $M$ with particles. Define the\emph{ principal curvatures at a singular point} $x\in S$ as the limit of the principal curvatures as the regular point converges to $x$. Now we give the following result by applying the maximum principle and the above two lemmas.

    \begin{lemma}\label{lm:injective}
    Let $M$ be a hyperbolic end with particles. Assume that $S_1$ and $S_2$ are two strictly concave surfaces in $M$ such that the supremum of the Gauss curvature over all the points on $S_1$ is less than the infimum of the Gauss curvature over all the points on $S_2$, and the limits of both principal curvatures at singular points on $S_1$ (resp. $S_2$) are the same. Then $S_2$ is strictly on the concave side of $S_1$.
    \end{lemma}

    \begin{proof}
    Denote by $\lambda_i$, $\mu_i$ the principal curvatures of $S_i$ for $i=1, 2$. Denote $C_{1}=\sup_{x\in S_1}\lambda_1(x)\mu_1(x)$ and $C_{2}=\inf_{x\in S_2}\lambda_2(x)\mu_2(x)$.  By assumption, we have $C_1<C_2$, and the Gauss-Bonnet formula shows that $C_2<1$.

  Suppose that $S_{2}$ is not strictly on the concave side of $S_{1}$. Note that $S_{1}$ and $S_{2}$ are both concave, therefore there exist point of $S_2$ where the $\partial_{\infty}M$-directed orthogonal geodesic rays from $S_{2}$ intersect the part of $S_{1}$ on the concave side exactly once. Consider $\psi^{t}:S_{2}\rightarrow M$ defined by $\psi^{t}(x)=\exp_{x}(t\cdot n_x)$, where $n_x$ is the $\partial_{\infty }M$-directed unit normal vector at $x$ of $S_{2}$ in $M$. Let $t_0=\sup\{t>0: \psi^{t}(x)\in S_1$ for some $x\in S_{2}\}$ and let $S^{t_0}_{2}=\psi^{t_0}(S_{2})$. Since $S_1$ and $S_2$ are both compact, then $t_0$ is attained at a point $x_0\in S_2$.  
It follows from Lemma \ref{the property of the pushing surface} that $S^{t_0}_2$ is a concave surface which intersects $S_{1}$ at a point $y_0=\psi^{t_0}(x_0)$, and it stays on the concave side of $S_{1}$. Denote by $\lambda^{t_0}_2$, $\mu^{t_0}_2$ the principal curvatures of $S^{t_0}_{2}$.

   If $y_0$ is a regular point, combining Theorem \ref{Maximum principle for regular points} and Lemma \ref{the property of the pushing surface}, we have
  \begin{equation}\label{ineq:injective}
 C_2\leq (\lambda_2\mu_2)(x_0)
  \leq(\lambda^{t_0}_2\mu^{t_0}_2)(y_0)\leq(\lambda_1\mu_1)(y_0)
  \leq C_{1}.
  \end{equation}
  This contradicts that $C_1<C_2$.

  If $y_0$ is a singular point, note that $S_1$ and $S_2$ behave  ``umbilically'' at singular points, and it follows from Statement (2) of Lemma \ref{the property of the pushing surface} that $S^{t_0}_{2}$ has an ``umbilical" point at $y_0$. Applying Lemma \ref{the property of the pushing surface} and Lemma \ref{comparison between principal curvatures at singularities} we have the same inequality \eqref{ineq:injective}.
  This contradicts again that $C_1<C_2$. Therefore, $S_{2}$ is strictly on the concave side of $S_{1}$.
  \end{proof}

     Using a similar argument as Lemma \ref{lm:injective}, we have the following proposition.

    \begin{proposition}
    \label{curvatures and positions of surfaces}
        Let $S_{i}, i=1,2$ be  concave surfaces of constant curvature $K_{i}\in(-1,0)$ in a hyperbolic end $M$ with particles for $i=1,2$. Then we have the following statements:
        \begin{enumerate}[(1)]
            \item $K_{1}<K_{2}$ if and only if $S_{2}$ is strictly on the concave side of $S_{1}$.
            \item $K_{1}=K_{2}$ if and only if $S_{1}$ coincides with $S_{2}$.
        \end{enumerate}
        \end{proposition}

        \begin{proof}
         \emph{Proof of Statement (1)}: First we show that $K_{1}<K_{2}$ implies that $S_{2}$ is strictly on the concave side of $S_{1}$.  Note that $K_{1}<K_{2}$ and the constant curvature surfaces $S_1$, $S_2$ behave ``umbilically" at singular points. This statement follows directly from Lemma \ref{lm:injective}.

         Now we prove the sufficiency, that is, if $S_{2}$ is strictly on the concave side of $S_{1}$, then $K_{2}>K_{1}$. Denote $S_{1}^{t}=\psi^{t}(S_{1})$. Set
         $\delta_{0}=\sup\{d(z,S_{2}): z\in S_{1}\}$. Obviously, $\delta_{0}>0$. Assume that $\delta_0$ is attained at a point $z_0\in S_1$ and denote $w_0=\psi^{\delta_0}(z_0)\in S_2\cap S_1^{\delta_0}$. Discussing $w_0$ in two cases (as a regular or singular point) as Lemma \ref{lm:injective} again, we have
         \begin{equation*}
            \begin{split}
                &\lambda_{1}^{\delta_{0}}(w_{0})\mu_{1}^{\delta_{0}}(w_{0})
                >\lambda_{1}(z_{0})\mu_{1}(z_{0})
                =1+K_{1},\\
               &\lambda_{1}^{\delta_{0}}(w_{0})\mu_{1}^{\delta_{0}}(w_{0})\leq \lambda_{2}(w_{0})\mu_{2}(w_{0})
               =1+K_{2}.
            \end{split}
         \end{equation*}
         Thus $K_{2}>K_{1}$.

        \emph{Proof of Statement (2)}: The sufficiency is obvious. Now we show the necessity.
        By assumption, $K_{1}=K_{2}$. Set $d_{1}=\sup\{d(x,S_{1}):x\in S_{2}$ is on the concave side of $S_{1}$ (including $S_1)\}$ and $d_{2}=\sup\{d(x,S_{2}):x\in S_{1}$ is on the concave side of $S_{2}$ (including $S_2)\}$. Note that $S_{1}=S_{2}$ if and only if $d_{1}=d_{2}=0$.

        If $d_{1}>0$, consider the surface $S_{1}^{d_{1}}$ obtained by pushing $S_{1}$ along orthogonal geodesics in a distance $d_{1}$ in the positive direction. Using the argument as above, we obtain the contradiction that $K_{1}<K_{2}$. This implies that $d_{1}=0$.

        If $d_{2}>0$,  consider the surface $S_{2}^{d_{2}}$ obtained by pushing $S_{2}$ along orthogonal geodesics in a distance $d_{2}$ in the positive direction. Using the same argument as above, we obtain the contradiction that $K_{1}>K_{2}$. This implies that $d_{2}=0$. Therefore, $S_{1}=S_{2}$.
    \end{proof}

    \begin{proposition}\label{injectivity of the map}
        For any $K\in(-1,0)$, the map $\phi_{K}:\mathcal{T}_{\Sigma,\theta}\times \mathcal{T}_{\Sigma,\theta}\rightarrow \mathcal{HE}_{\theta}$ is injective.
    \end{proposition}

    \begin{proof}
         Assume that $(h,h')$, $(h_{1},h_{1}')\in\mathcal{T}_{\Sigma,\theta}\times \mathcal{T}_{\Sigma,\theta}$ satisfy that $\phi_{K}(h,h')=\phi_{K}(h_{1},h_{1}'):=M$. Then $M$ contains a concave surface $S$ of constant curvature $K$, with the induced metric $I=(1/|K|)h$ and the third fundamental form $\III=(1/|K^{*}|)h'$, and also contains a concave surface $S_{1}$ of constant curvature $K$, with the induced metric $I_{1}=(1/|K|)h_{1}$ and the third fundamental form $\III=(1/|K^{*}|)h_{1}'$.
         By Proposition \ref{curvatures and positions of surfaces}, we have $S=S_{1}$. Then $h=h_{1}$ and $h'=h_{1}'$, which implies that $(h,h')=(h_{1},h_{1}')$.
    \end{proof}

    \subsection{The continuity of the map $\phi_K$} The map $\phi_K$ relates deeply to the minimal Lagrangian maps between two hyperbolic surfaces with cone singularities in $\mathcal{T}_{\Sigma,\theta}$, which provides the embedding data to construct a hyperbolic end with particles. With the result in \cite[Lemma 3.19]{CS}, we have the following proposition.
  \begin{proposition}\label{Continuity of the map}
        For any $K\in(-1,0)$, the map $\phi_{K}:\mathcal{T}_{\Sigma,\theta}\times \mathcal{T}_{\Sigma,\theta}\rightarrow \mathcal{HE}_{\theta}$ is continuous.
    \end{proposition}

    \begin{proof}
         It suffices to prove that if the sequence $(h_{k},h_{k}')_{k\in \N}$ converges to $(h,h')\in\mathcal{T}_{\Sigma,\theta}\times \mathcal{T}_{\Sigma,\theta}$, then the sequence $(\phi_{K}(h_{k},h_{k}'))_{k\in \N}$ converges to $\phi_{K}(h,h')\in\mathcal{HE}_{\theta}$. Denote by $m_{k}$ the unique minimal Lagrangian map between $(\Sigma, h_{k})$ and $(\Sigma, h_{k}')$ isotopic to the identity and by $m$ the unique minimal Lagrangian map between $(\Sigma, h)$ and $(\Sigma, h')$ isotopic to the identity.

        By the proof in \cite[Lemma 3.19]{CS},  the sequence $(m_{k})_{k\in \N}$ converges $m$. Let $b_{k}:T\Sigma\rightarrow T\Sigma$ be the bundle morphism defined outside the singular locus which is described in Proposition \ref{Toulisse2} with the property  $m_{k}^{*}(h_{k}')=h_{k}(b_{k}\bullet,b_{k}\bullet)$. Then $b_{k}$ converges to a bundle morphism from $T\Sigma$ to $T\Sigma$, say $b$.

        Let $I_{k}=(1/|K|)h_k$ and $B_{k}=\sqrt{1+K}b_k$. Then $(\Sigma, I_{k},B_{k})_{k\in \N}$ converges to $(\Sigma,I,B)$, in the sense that $I_k$ and $B_k$ converge to $I=(1/|K|)h$ and $B=\sqrt{1+K}b$, respectively. This implies that $(\phi_{K}(h_{k},h_{k}'))_{k\in \N}$ converges to $\phi_{K}(h,h')$ in $\mathcal{HE}_{\theta}$. The lemma follows.
    \end{proof}

    \subsection{The properness of the map $\phi_K$}
    To prove this property of $\phi_K$, we first give a comparison between the lengths of closed geodesics in the same isotopy class on the metric boundary $\partial_{0}M$ and on a strictly concave surface in a hyperbolic end $M$ with particles.

    \begin{lemma}
        \label{comparison between lengths}
            Let $M$ be a hyperbolic end with particles. Let $S$ be a strictly concave surface in $M$. Then for any closed geodesic $\gamma$ on $\partial_0M$, the length of $\gamma$ is smaller than the length of any closed minimizing geodesic $\gamma'$ on $S$ isotopic to $\gamma$ in $M$.
    \end{lemma}

    \begin{proof}
  Let $r: M\rightarrow \partial_0M$ be the closest point projection of $M$ to the metric boundary $\partial_0M$ (this is well-defined since $\partial_0M$ is concave). Note that if $x\in M$ is a singular point, then the closet point projection is along the singular curve through $x$. Then $r$ is 1-Lipschitz with respect to the hyperbolic metric on $M$ and the induced metric on $\partial_0M$. Therefore, the marked length spectrum of $\partial_0 M$ is bounded by the marked length spectrum of $S$. This completes the proof.
    \end{proof}

    Let $X$ be a topological space and let $(x_n)_{n\in\mathbb{N}}$ be a sequence of elements in $X$. We say that  $(x_n)_{n\in\mathbb{N}}$ \emph{tends to infinity} if $(x_n)_{n\in\mathbb{N}}$ is not contained in any compact subset of $X$. 

    Now we recall a result in Teichm\"uller spaces of hyperbolic surfaces with cone singularities of prescribed angles less than $\pi$. This follows from an analysis on the parametrization of $\mathcal{T}_{\Sigma,\theta}$ by Fenchel-Nielsen coordinates associated to a fixed pants decomposition and the Collar lemma for hyperbolic cone-surfaces (see \cite[Theorem 3]{DP}).

  \begin{lemma}\label{lm:infinity}
 Let $(h_n)_{n\in\mathbb{N}}$ be a sequence of elements in $\mathcal{T}_{\Sigma,\theta}$.  Then the following two statements are equivalent:
  \begin{enumerate}[(1)]
    \item $(h_n)_{n\in\mathbb{N}}$ tends to infinity.
    \item For any $k\in\mathbb{N^{+}}$, there exists a simple closed curve $\gamma_{k}$ on $\Sigma$ and an integer $N>0$ (depending on $k$ and $\gamma_{k}$), such that $\ell_{\gamma_{k}}(h_N)<(1/k)\,\ell_{\gamma_{k}}(h_0)$.
  \end{enumerate}
  \end{lemma}

 \begin{proposition}\label{properness of the map}
         For any $K\in(-1,0)$, the map $\phi_{K}:\mathcal{T}_{\Sigma,\theta}\times \mathcal{T}_{\Sigma,\theta}\rightarrow  \mathcal{HE}_{\theta}$ is proper.
    \end{proposition}

\begin{proof}

Denote $\phi_K(h_n, h'_n)=(M_n, g_n)$ for $n\in\mathbb{N}$. We suppose that $(M_n,g_n)_{n\in \N}$ converges to a limit $(M,g)$, and will prove that $(h_n)_{n\in \N}$ and $(h'_n)_{n\in \N}$ must remain bounded.

It follows from the hypothesis that $(m_n)_{n\in \N}$ and $(l_n)_{n\in \N}$ remain bounded, where $m_n$ and $l_n$ are the induced metric and measured bending lamination on 
$\partial_0M_n$ for $g_n$. After extracting a subsequence, we can suppose that $(m_n)_{n\in \N}$ converges to a limit $m$, and $(l_n)_{n\in \N}$ converges to a limit $l$  (see Proposition \ref{prop: homeomorphism}), where $m$ and $l$ are the induced metric and measured bending lamination on $\partial_0M$ for $g$.

Note that the concave surface $\Sigma_{K,n}$ of constant curvature $K$ in $M_n$ has the induced metric $I_n=(1/|K|)h_n$. It follows from Lemma \ref{comparison between lengths} that 
$\ell_{\gamma}(m_n)<\ell_{\gamma}(I_n)=(1/\sqrt{|K|})\ell_{\gamma}(h_n)$
 for all simple closed curves $\gamma$ on $\Sigma$. Suppose that $(h_n)_{n\in\mathbb{N}}$ is not bounded. Combined with Lemma \ref{lm:infinity}, this shows that, for any $k>0$, there exists a simple closed curve $\gamma_{k}$ on $\Sigma$ and an integer $N>0$ (depending on $k$ and $\gamma_{k}$), such that $\ell_{\gamma_{k}}(m_N)<(k\sqrt{|K|})^{-1}\ell_{\gamma_{k}}(h_0)$. 
Applying Lemma \ref{lm:infinity} again, we find that $(m_n)_{n\in\mathbb{N}}$ tends to infinity, which leads to a contradiction.

We first note that there exists $r>0$ such that for all $x\in \Sigma_{K,n}$, the distance from $x$ to 
$\partial_0 M_n$ is at most $r$. Otherwise, there would be a sequence $(x_n)_{n\in \N}$ with $x_n\in \Sigma_{K,n}$ and 
$d_{g_n}(x_n,\partial_0M_n)\to \infty$, and this would contradict the fact that 
$I_n$, $m_n$ are converging to metrics of constant curvature, $l_n$ is converging to $l$, and the area of a concave surface in $M_n$ expands exponentially with respect to the distance $r$ along the normal flow starting from $\partial_0 M_n$.

Let $S_{r,n}$ be the set of points at distance $r$ from 
$\partial_0M_n$  for $g_n$, with the induced metric $I_{r,n}$. For all $n$, $S_{r,n}$ is a 
smooth (outside the singular locus), strictly 
 concave surface. Let $\III_{r,n}$ be 
 the third fundamental form of $I_{r,n}$. Notice that since $m_n\to m$ and $l_n\to l$, $\III_{r,n}$ must also converge to a limit $\III_{r}$.

We claim that the length spectrum of 
the third fundamental form $\III_n$ of $\Sigma_{K,n}$ is smaller than the length spectrum of 
$\III_{r,n}$ of $S_{r,n}$. This is equivalent to proving that the length spectrum of the induced metric on the dual surface $\Sigma^d_{K,n}$ in $M^d_n$, the GHM de Sitter spacetime with particles dual to $M_n$ as seen in Section \ref{sc:duality},
is smaller than the length spectrum of the induced metric on the surface $S^d_{r,n}$ dual to $S_{r,n}$. To prove this dual statement, note that the definition of the duality
shows that $S^d_{r,n}$ is the set of points a distance $r$ from the initial singularity $(\partial_0M_n)^d$ of $M_n^d$. As a consequence, the open segments of length $r$ orthogonal to $S_{r,n}^d$ in the past foliate the past of $S_{r,n}^d$, and the de Sitter metric on the past of $S_{r,n}^d$ can be written as
$$ -dt^2 + I^d_{t,n}, ~t\in (0,r)~, $$
where $I^d_{t,n}$ is the induced metric on $S_{t,n}^d$ and therefore isometric to $\III_{t,n}$.

Since the $S_{t,n}^d$ are future-convex, $I^d_{t,n}$ is increasing in $t$, and therefore $I^d_{t,n}\leq I^d_{r,n}$ for all $t\leq r$. It follows that the induced metric on the surface $\Sigma^d_{K,n}$ can be written as
$$ I^d_{n} = -dt^2+I^d_{t,n} \leq I^d_{t,n}\leq I^d_{r,n}~, $$
where we are using the identification between $\Sigma_{K,n}^d$, $S^d_{t,n}$ and $S^d_{r,n}$ through the normal flow of the $(S^d_{t,n})_{t\in (0,r)}$. Here $t$ is the function defined on $\Sigma^d_{K,n}$ as the distance to the initial singularity of $M^d_n$.

We have now established that the length spectrum of $\III_{n}$ is smaller than that of $\III_{r,n}$, and so uniformly bounded. This shows that, after extracting a subsequence, $(\III_n)_{n\in \N}$ converges to a limit. 
Recall that in Lemma \ref{definition of the map} we showed that $h'_n=|K^*|\III_n$, where $K^*=K/(1+K)$. Therefore, $(h'_n)_{n\in \N}$ also converges to a limit.
 \end{proof}

\emph{Proof of Proposition \ref{prop:parametrization map}}.
By Proposition \ref{prop: homeomorphism}, $\mathcal{HE}_{\theta}$ is homeomorphic to $T^{*}\mathcal{T}_{\Sigma,\theta}$. Therefore $\mathcal{T}_{\Sigma,\theta}\times\mathcal{T}_{\Sigma,\theta}$ and $\mathcal{HE}_{\theta}$ are both simply connected. Note that $\mathcal{T}_{\Sigma,\theta}\times\mathcal{T}_{\Sigma,\theta}$ and $\mathcal{HE}_{\theta}$ have the same dimension and have no boundary. Combined with Proposition \ref{injectivity of the map}, Proposition \ref{Continuity of the map}, and Proposition \ref{properness of the map}, it follows that $\phi_{K}$ is a homeomorphism.

  \subsection{The convergence of K-surfaces}
  Fix a hyperbolic end $M$ with particles. By Proposition \ref{prop:parametrization map}, $M$ contains a locally concave surface $S_K$ of constant curvature $K$ for all $K\in(-1,0)$ (since $\phi_K$ is surjective). Furthermore, the constant curvature $K$-surface in $M$ is unique (since $\phi_k$ is injective) and distinct constant curvature $K$-surfaces are disjoint from each other (this follows from Proposition \ref{curvatures and positions of surfaces}).

  To show that $M$ admits a foliation by locally concave constant curvature surfaces, it suffices to prove that the union of constant curvature $K$-surfaces $S_K$ over all $K\in(-1,0)$ is exactly $M$. In particular, we show that the sequence $(S_{K_n})_{n\in\mathbb{N}}$ of constant curvature $K_n$-surfaces in $M$ converges to $S_{K}$ in the $C^2$-topology if $K_n\rightarrow K\in(-1,0)$.

  Note that the singularities on a constant curvature surface in $M$ behave like ``umbilical'' points and the cone angles are less than $\pi$, the theorem given by F. Labourie  \cite[Theorem D]{Lab89} (which describes a degenerating phenomenon of a sequence of isometric embedding of a surface with the determinants of second fundamental form bounded below by $\varepsilon>0$ in a Riemannian 3-manifold with sectional curvature less than $K_0$ for a real number $K_0$) can be generalized to the following case of hyperbolic ends with cone singularities.

\begin{theorem}\label{thm:degenerate}
Let $M$ be a hyperbolic end with particles and let $S_{n}$ be a sequence of surfaces in $M$ with the determinants of second fundamental forms bounded below by $\varepsilon>0$, with the induced metric $g_n$. Let $f_n$ be an embedding of the prescribed surface $\Sigma$ into $M$ with the image $f_n(\Sigma)=S_n$. Assume that $f_n^{*}(g_n)$ converges to a Riemannian metric $g_{\infty}$ in the $\mathcal{C}^2$-topology, and $f_n$ converges to an embedding $f_{\infty}:\Sigma\rightarrow M$ in the $\mathcal{C}^{0}$-topology but not in the $\mathcal{C}^{3}$-topology (outside the singular locus), then there exists a complete geodesic $\gamma$ of $(\Sigma, g_{\infty})$ such that $f_{\infty}|_{\gamma}$ is an isometry from $\gamma$ into a geodesic of $M$.
\end{theorem}

  \begin{lemma}\label{complete geodesics}
    Let $M$ be a hyperbolic manifold with particles which has a concave metric boundary. Assume that $\bar{M}$ contains a complete geodesic $\gamma$ which stays in a bounded distance from $\partial_0M$, then $\gamma$ lies on the metric boundary $\partial_0M$.
    \end{lemma}

     \begin{proof}
    Consider a function $u:\gamma\rightarrow \mathbb{R}_{\geq 0}$ defined by
    \begin{equation*}
    u(x)=\sinh d(x,\partial_0M).
    \end{equation*}
    Denote by $g$ the metric on $\bar{M}$. It is known that $u$ satisfies the equality $\Hess(u)\geq ug$ in the distributional sense (see e.g. \cite[Lemma A.12]{MoS}), since $\partial_0M$ is concave and the map 
    $\exp:  N\partial_0M\rightarrow M$  is a homeomorphism (see Lemma \ref{lm:exp}). Assume that $\gamma$ is a geodesic parameterized by arclength, then $(u\circ \gamma)''\geq u\circ\gamma$. Note that $\gamma$ stays at bounded distance from $\partial_0M$.  Applying the maximum principle, we obtain that $u\circ\gamma=0$ for all $t\in\mathbb{R}$. Therefore, the complete geodesic $\gamma$ lies on the metric boundary $\partial_0M$.
    \end{proof}

\begin{lemma}\label{lm:convergence of k-surfaces}
Let $(M,g)$ be a hyperbolic end with particles. Let $(S_{K_n})_{n\in\mathbb{N}}$ be a sequence of locally concave surfaces in $M$ of constant curvature $K_n\in(-1,0)$. Then the following statements hold.

\begin{enumerate}[(1)]
  \item If $K_n\rightarrow K\in[-1,0)$ with $K_n\not =K$ for any $n\in\mathbb{N}$, then the sequence $(S_{K_n})_{n\in\mathbb{N}}$ converges to $S_K$ in the compact-open topology (or $C^0$-topology). Moreover, if $K\in(-1,0)$, then the sequence $(S_{K_n})_{n\in\mathbb{N}}$ converges to $S_K$ in the $\mathcal{C}^2$-topology (outside the singular locus).
  \item If $K_n\rightarrow0$, then the (least) distance from the surface $S_{K_n}$ to the metric boundary $\partial_0M$ tends to infinity as $n\rightarrow \infty$.
\end{enumerate}
\end{lemma}

\begin{proof}
\emph{Proof of Statement (1)}:
Denote by $\Phi$ the $\partial_{\infty}M$-directed normal flow,  given by the exponential map $\exp: N\partial_0M\rightarrow M$  (see the map $\exp$ in Lemma \ref{lm:exp}). By the  Gauss-Bonnet formula for surfaces with cone singularities (see e.g. \cite[Propositon 1]{Troyanov}), the area of $S_{K_n}$ is equal to $(2\pi/K_{n})\,\chi(\Sigma,\theta)$, where
 \begin{equation*}
 \chi(\Sigma,\theta)=\chi(\Sigma)+\sum\limits_{i=1}^{n_0}(\theta_{i}/2\pi-1)<0.
 \end{equation*}
 Therefore, $Area(S_{K_n})\rightarrow (2\pi/K)\,\chi(\Sigma,\theta)=Area(S_K)$ as $n\rightarrow \infty$, where $K\in[-1,0)$.

 We claim that $S_{K_n}$ converges to $S_K$ in the compact-open topology as $n\rightarrow \infty$.  Indeed, we first fix an embedding map $f_{\infty}:\Sigma\rightarrow M$ such that $f_{\infty}(\Sigma)=S_{K}$. Then let $f_n:\Sigma\rightarrow M$ be the embedding map compatible with the flow $\Phi$, that is, the map $f_{n}\circ f_{\infty}^{-1}: S_K\rightarrow S_{K_n}$ coincides with the homeomorphism from $S_K$ to $S_{K_n}$ induced by the flow $\Phi$ for all $n\in\mathbb{N}$.  Suppose that there exists a compact subset $U\subset\Sigma$, such that the sequence $(f_n(U))_{n\in\mathbb{N}}$ does not converge to $f_{\infty}(U)$ in $M$. Then there exists a neighborhood $V$ of $f_{\infty}(U)$ in $M$ such that we can find a subsequence $(f_{n_k})_{k\in\mathbb{N}}$ with $f_{n_k}(U)$ disjoint from $V$ for all $k\in\mathbb{N}$. By Proposition \ref{curvatures and positions of surfaces}, there exists an integer $N>0$, such that $f_{n}(U)$ is disjoint from $V$ for $n\geq N$, and $S_{K_n}$ is disjoint from $S_K$ for all $K_n\not =K$. Combined with the construction of $f_n$,  the distance from $f_{\infty}(U)\subset S_K$ to $f_n(U)\subset S_{K_n}$ along the flow $\Phi$ is bigger than a positive number $r_0$ for all $n\geq N$. Note that the induced metric by $M$ is strictly increasing along the normal flow $\Phi$. This implies that the sequence $(|Area(S_K)-Area(S_{K_n})|)_{n\in\mathbb{N}^+}$ does not converge to zero, which leads to a contradiction.

 Now we show that $(S_{K_n})_{n\in\mathbb{N}}$ converges to $S_K$ in the $C^2$-topology for all $K\in(-1,0)$. Denote by $g_n$ the induced metric on $S_{K_n}$ for all $n\in\mathbb{N}$. Note that $S_{K_n}$ is orthogonal to the singular lines $l_{k}$ (which are homeomorphic to $\{p_{k}\}\times \mathbb{R}$) and the angle of the singularity on $S_{K_n}$ at the intersection with $l_{k}$ is $\theta_{k}\in(0,\pi)$ for $k=1,...,n_0$.
        Therefore, the metrics $g_{n}$ can be written as follows:
        \begin{equation*}
        g_{n}=(1/|K_{n}|)\widehat{g_{n}},
        \end{equation*}
        where $\widehat{g_{n}}\in\mathfrak{M}^{\theta}_{-1}$ for all $n\in\mathbb{N}^+$.

         For convenience, we assume that $S_{K_0}=\partial_0M$, that is, $K_0=-1$. By Lemma \ref{comparison between lengths},
         for any simple closed curve $\gamma$ on $\Sigma$,
         we have
         \begin{equation*}
         \ell_{f_{n}(\gamma)}(g_{i})\geq\ell_{f_{0}(\gamma)}(g_{0}),
         \end{equation*}
         for all $n\in\mathbb{N}$.
         Note that $K_{n}$ converges to $K\in(-1, 0)$. Then
         \begin{equation*}
            \ell_{f_{n}(\gamma)}(\widehat{g_{n}})
         =\ell_{f_{n}(\gamma)}(|K_{n}|g_{n})
         =\sqrt{|K_{n}|}\ell_{f_{n}(\gamma)}(g_{n})
         \geq\sqrt{|K|}\,\ell_{f_{0}(\gamma)}(g_{0})
         =\sqrt{K/K_{0}}\,\ell_{f_{0}(\gamma)}(\widehat{g_{0}}),
         \end{equation*}
         for all $n\in\mathbb{N}$. Here $K/K_{0}<1$.

         Denote by $f_{n}^{*}(\widehat{g_{n}})$ the pull-back metric on $\Sigma$
         of $\widehat{g_{n}}$ under $f_{n}$ and still denote by $f_{n}^{*}(\widehat{g_{n}})$ its isotopy class in $\mathcal{T}_{\Sigma,\theta}$ for all $n\in\mathbb{N}$.
         For any simple closed curve $\gamma$ on $\Sigma$, we get
         \begin{equation*}
         \ell_{\gamma}(f_{n}^{*}(\widehat{g_{n}})) \geq\sqrt{K/K_{0}}\,\ell_{\gamma}({f_{0}}^{*}(\widehat{g_{0}})).
         \end{equation*}

         By Lemma \ref{lm:infinity}, the set $\{f_{n}^{*}(\widehat{g_{n}}):n\in\mathbb{N}\}$ is compact in $\mathcal{T}_{\Sigma,\theta}$. Therefore, up to extracting a subsequence, $(f_{n}^{*}(\widehat{g_{n}}))_{n\in \N}$ converges in $\mathcal{T}_{\Sigma,\theta}$. Note that $(f_n)_{n\in\mathbb{N}}$ is compatible with the flow $\Phi$. $(f_{n}^{*}(\widehat{g_{n}}))_{n\in \N}$ converges to
         $f_{\infty}^{*}(\widehat{g_K})$ in the $\mathcal{C}^{2}$-topology (outside the singular locus), where $\widehat{g_K}=|K|\,g_K$ and $g_K$ is the induced metric on $S_K$ in $M$. In particular, $f_{n}^{*}(g_{n})$ converges to $g_{\infty}=f_{\infty}^{*}(g_K)$ in the $\mathcal{C}^{2}$-topology (outside the singular locus). Note that $\Sigma$ is compact and by the above result we have $f_n$ converges to $f_{\infty}$ in the $\mathcal{C}^0$-topology.

        We claim that $f_n$ converges to $f_{\infty}$ in the $\mathcal{C}^{3}$-topology. Otherwise, it follows from Theorem \ref{thm:degenerate} that there exists a complete geodesic $\gamma$ of $(\Sigma, g_{\infty})$ such that $f_{\infty}|_{\gamma}$ is an isometry from $\gamma$ into a geodesic of $(M,g)$. Note that the geodesic 
        $f_{\infty}(\gamma)$ lies on $S_K$ and thus stays in a bounded distance from $\partial_0M$. Combined with Lemma \ref{complete geodesics}, 
        $f_{\infty}(\gamma)$ is contained in $\partial_0M$ which is disjoint from $S_K$. This leads to a contradiction. Therefore, Statement (1) follows.

 \emph{Proof of Statement (2)}: We first fix the surface $S_{K_1}$ and denote $S=S_{K_1}$. Consider a map $\psi^{t}:S\rightarrow M$ defined by $\psi^{t}(x)=\exp_{x}(t\cdot n_x)$, where $n_x$ is the $\partial_{\infty }M$-directed unit normal vector at $x$ of $S$ in $M$. For any $T>0$, we denote $S^{T}=\psi^T(S)$ and denote by $\lambda^{T}$, $\mu^{T}$ the principal curvatures of $S^{T}$. By Lemma \ref{the property of the pushing surface}, the principal curvatures of $S^{T}$ are
  \begin{equation*}
  \lambda^{T}(\psi^{T}(x))
  =\frac{\lambda(x)+\tanh (T)}{1+\lambda(x)\tanh (T)},\qquad
  \mu^{T}(\psi^{T}(x))
  =\frac{\mu(x)+\tanh (T)}{1+\mu(x)\tanh (T)},
  \end{equation*}
  where $\lambda(x)$ and $\mu(x)$ are the principal curvatures of $S$ at $x$.

  Let $C^{T}=\sup_{y\in S^{T}}\lambda^{T}(y)\mu^{T}(y)$. Then $\lambda^{t}(\psi^{t}(x))\mu^{t}(\psi^{t}(x))$ increasingly tends to 1 as $t\rightarrow +\infty$ for all $x\in S$. Note that $S^{T}$ is also locally concave and compact, so $C^{T}\in(0,1)$.

  By assumption, $K_n\rightarrow 0$. Therefore there exists $N_{T}>0$ (depending only on $T$) such that for all $n\geq N_{T}$, we have
  \begin{equation*}
  -1+C^T<K_n<0.
  \end{equation*}

  Note that $S^{T}$ and $S_{K_n}$ have constant curvature and behave ``umbilically" at singular points. It follows from  Lemma \ref{lm:injective} that $S_{K_n}$ is on the concave side of $S^{T}$ for all $n\geq N_T$. Observe that $C_T\rightarrow 1$ as $T\rightarrow +\infty$, and the distance from $S^T$ to $\partial_0 M$ tends to infinity as $T\rightarrow +\infty$. Combined with the result above, the distance from $S_{K_n}$ to $\partial_0 M$ tends to infinity as $n\rightarrow \infty$.
\end{proof}

The following corollary is a direct consequence of Proposition \ref{curvatures and positions of surfaces} and Lemma \ref{lm:convergence of k-surfaces}.

\begin{corollary}\label{k-foliation}
Let $M$ be a hyperbolic end with particles. Then the union of the constant curvature $K$-surfaces $S_K$ in $M$ over all $K\in(-1,0)$ provides a $\mathcal{C}^2$-foliation of the regular part of $M$.
\end{corollary}

\emph{Proof of Theorem \ref{tm:foliation}}. As discussed in the beginning of Section 4.5, it follows directly from Proposition \ref{prop:parametrization map} and Corollary \ref{k-foliation}.

\subsection{Applications to smooth grafting}
\label{ssc:SGr}
In the non-singular case, the landslide flow is defined in \cite{BMS1} as a map $L:S^1\times \cT\times \cT\to \cT\times \cT$, sending $(e^{i\alpha}, h, h^*)$ to the left and right metrics of the unique GHM AdS spacetime containing a constant curvature surface with induced metric $\cos^2(\alpha/2)h$ and third fundamental $\sin^2(\alpha/2)h^*$.

It is also proved there that the landslide map, composed with the canonical projection on the first factor, has a complex extension as the ``smooth grafting'' map $sgr:(0,1)\times \cT\times \cT\to \cT$, sending $(r,h,h^*)$ to the conformal metric at infinity of the unique hyperbolic end containing a constant curvature surface with induced metric $\frac{(1+r)^2}{4r}h$ and third fundamental form $\frac{(1-r)^2}{4r}h^*$. This surface has constant curvature $-4r/(1+r)^2$. The map $sgr$ is obtained from another grafting map $SGr:(0,1)\times \cT\times \cT\to \cCP$ by composition on the left with the forgetful map from $\cCP$ to $\cT$.

The landslide map limits in a precise sense to the earthquake map $\cT\times \cML\to \cT$, while the smooth 
grafting map limits in a precise sense to the grafting map $\cT\times \cML\to \cT$.

The results of \cite{CS} on constant Gauss foliations in convex GHM AdS spacetimes with particles lead to an extension of the landslide flow to hyperbolic surfaces with cone singularities of angles less than $\pi$. In the same manner, the results presented here on constant curvature foliations of hyperbolic ends with particles lead directly, by extending the arguments of \cite{BMS1} without any serious change, to the definition of the smooth grafting maps $sgr_{\theta}:(0,1)\times \cT_{\Sigma,\theta}\times \cT_{\Sigma,\theta}\to \cT_{\Sigma,\theta}$ and $SGr_{\theta}:(0,1)\times \cT_{\Sigma,\theta}\times \cT_{\Sigma,\theta}\to \cCP_{\theta}$.

It can be proved, using the same arguments as in \cite{BMS1}, that:
\begin{enumerate}
\item The smooth grafting map $sgr$ provides a complex extension of the landslide map. More precisely, if $L^1:S^1\times \cT_{\Sigma,\theta}\times \cT_{\Sigma,\theta}\to \cT_{\Sigma,\theta}$ is the landslide map followed by projection on the first factor, then the ``complex landslide'' map:
  \begin{eqnarray*}
    D \times \cT_{\Sigma,\theta}\times \cT_{\Sigma,\theta} & \to & \cT_{\Sigma,\theta} \\
    (re^{i\alpha}, h, h^*) & \mapsto & sgr(r,L_{e^{i\alpha}}(h,h^*)) \\
  \end{eqnarray*}
defines a holomorphic map from the unit disk $D$ to $\cT_{\Sigma,\theta}$ extending $L^1$ to the unit disk, for any fixed $h$ and $h^*$.
\item The smooth grafting maps $sgr_{\theta}$ and $SGr_{\theta}$ limit, in the same suitable sense as in \cite{BMS1}, to the grafting maps $gr_{\theta}:\cT_{\Sigma,\theta}\times \cML_{\mathfrak{p}}\to \cT_{\Sigma,\theta}$ and $Gr_{\theta}:\cT_{\Sigma,\theta}\times \cML_{\mathfrak{p}}\to \cCP_{\theta}$.
\end{enumerate}


%% file: foliation6.tex
\section{Foliations of de Sitter spacetimes with particles by constant curvature surfaces}

In this last section, we prove that convex GHM de Sitter spacetimes with particles admit a foliation by constant Gauss curvature surfaces orthogonal to the particles. As a consequence, for each $K^d\in (-\infty,0)$, the space of convex GHM de Sitter spacetimes with particles can be parameterized by the product of two copies of $\mathcal{T}_{\Sigma, \theta}$ in terms of constant curvature $K^d$-surface.

\subsection{Foliation of de Sitter spacetimes with particles by $K$-surfaces}

As a consequence of Proposition \ref{pr:dual}, each foliation of a non-degenerate hyperbolic end with particles has a dual foliation of the dual future-complete convex GHM de Sitter space-time with particles.

Observe that the curvature $K^d$ varies from $-\infty$ to $0$ in Proposition \ref{pr:dual}, combined with Theorem \ref{tm:foliation}, we therefore obtain Corollary \ref{corollary:k-foliation of dS}, which states that every future-complete convex GHM de Sitter spacetime $M^d$ with particles admits a unique foliation by surfaces of constant curvature $K^d$, with $K^d$ varying from $-\infty$ near the initial singularity to $0$ near the boundary at infinity. In particular, for each $K^d\in (-\infty,0)$, $M^d$ contains a unique closed surface of constant curvature $K^d$. Combined with Theorem \ref{tm:dual-surfaces} and Corollary \ref{k-foliation},  the union of the constant curvature $K^d$-surfaces in $M^d$ over all $K^d\in(-\infty,0)$ provides a $\mathcal{C}^2$-foliation of the regular part of $M^d$.

\subsection{A parametrization of $\cDS_{\theta}$ by $\mathcal{T}_{\Sigma, \theta}\times\mathcal{T}_{\Sigma,\theta}$}

We can also give a parametrization of $\cDS_{\theta}$ in terms of constant curvature surfaces.

Let $K^d\in(-\infty, 0)$ and let $(h,h')\in\mathfrak{M}^{\theta}_{-1}\times\mathfrak{M}^{\theta}_{-1}$ be a pair of normalized metrics. Using a similar argument as in Lemma \ref{definition of the map},   there exists a unique convex GHM de Sitter spacetime $M^d$ with particles which contains a surface of constant curvature $K^d$, with the induced metric $I^d=(1/|K^d|)h'$ and the third fundamental form $\III^d=(1/|K^{d*}|)h$, where $K^{d*}=K^d/(1-K^d)$.

For any $K^d\in(-\infty, 0)$, define the map $\psi_{K^d}:\mathcal{T}_{\Sigma, \theta}\times\mathcal{T}_{\Sigma,\theta}\rightarrow \mathcal{DS}_{\theta}$ by assigning to an element $(\tau,\tau')\in\mathcal{T}_{\Sigma, \theta}\times\mathcal{T}_{\Sigma,\theta}$ the isotopy class of the de Sitter spacetime with particles satisfying the above property. Combining Proposition \ref{prop:parametrization map} and 
the duality between strictly concave surfaces in a hyperbolic end $M$ with particles and strictly future-convex spacelike surfaces in the dual de Sitter spacetimes $M^d$ with particles (see Theorem \ref{tm:dual-surfaces}), it follows that the parametrization $\psi_{K^d}$ is 
equal to the composition map $\delta\circ\phi_K$, and therefore a homeomorphism (as shown in Figure \ref{fig:parametrization}).

\begin{figure}
$\xymatrix{
                & \mathcal{T}_{\Sigma, \theta}\times\mathcal{T}_{\Sigma,\theta} \ar[dr]^{\psi_{K^d}}      \ar[dl]_{\phi_{K}}       \\
 \cHE_{\theta}  \ar[rr]^{\delta} & &     \cDS_{\theta}       }$
  \caption{\small{A diagram showing the parametrizations of $\cHE_{\theta}$ and $\cDS_{\theta}$ by $\mathcal{T}_{\Sigma, \theta}\times\mathcal{T}_{\Sigma,\theta}$, respectively.}}
     \label{fig:parametrization}
\end{figure}
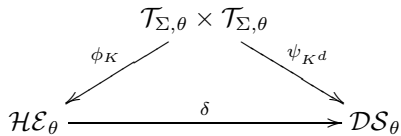